\documentclass[a4paper,12pt]{article}

\usepackage{amsthm, amsmath, amssymb }
\usepackage[ngerman, english]{babel}
\usepackage[utf8]{inputenc} 
\usepackage[T1]{fontenc} 
\usepackage{enumitem}
\usepackage{Mydefs}
\usepackage[
]{hyperref}
 
\title{Mittag-Leffler Analysis II: Application to the fractional heat equation}
\date{\today}
\author{Martin~Grothaus \and Florian~Jahnert}
 
\begin{document}

\maketitle

\begin{abstract}
Mittag-Leffler analysis is an infinite dimensional analysis with respect to non-Gaussian measures of Mittag-Leffler type which generalizes the powerful theory of Gaussian analysis and in particular white noise analysis. In this paper we further develop the Mittag-Leffler analysis by characterizing the convergent sequences in the distribution space. Moreover we provide an approximation of Donsker's delta by square integrable functions. Then we apply the structures and techniques from Mittag-Leffler analysis in order to show that a Green's function to the time-fractional heat equation can be constructed using generalized grey Brownian motion (ggBm) by extending the fractional Feynman-Kac formula \cite{Sch92}. Moreover we ana\-lyse ggBm, show its differentiability in a distributional sense and the existence of corresponding local times.
\end{abstract}



\newpage

\section{Introduction}
The Mittag-Leffler analysis is part of the field of research which tries to transfer the concepts and results known from Gaussian analysis to a non-Gaussian setting. First steps in this direction were made in \cite{Ito88} using the Poisson measure. This approach was generalized with the help of biorthogonal systems, called Appell systems, see e.g.~\cite{Da91, ADKS96, KSWY98}, which are suitable for a wide class of measures including the Gaussian measure and the Poisson measure \cite{KSS97}. Two properties of the measure $\mu$ are essential: An analyticity condition of its Laplace transform and a non-degeneracy or positivity condition (see also \cite{KK99}). In \cite{KSWY98} a test function space $(\cN)^1$ and a distribution space $(\cN)^{-1}_{\mu}$ is constructed and the distributions are characterized using an integral transform and spaces of holomorphic functions on locally convex spaces. The corresponding results in Gaussian analysis, see e.g.~\cite{KLPSW96, KLS96, GKS99}, paved the way for applying successfully the abstract theory to numerous concrete problems such as intersection local times for Brownian motion \cite{FHSW97} as well as for fractional Brownian motion \cite{DOS08, OSS11} and stochastic partial differential equations, see e.g.~\cite{GKU99, GKS00} or the monograph \cite{HOUZ96}.   

In Mittag-Leffler analysis, the measures $\mu_\beta$, $0<\beta<1$, are defined via their characteristic function given by Mittag-Leffler functions. It is shown in \cite{GJRS14} that the Mittag-Leffler measures belong to the class of measures for which Appell systems exist. Furthermore the (weak) integrable functions with values in the distribution space $(\cN)^{-1}_{\mu_\beta}$ are characterized and a distribution is constructed which is a generalization of Donsker's delta in Gaussian analysis.

We further develop the theory of Mittag-Leffler analysis in this work by characterizing the convergent sequences in the distribution space $(\cN)^{-1}_{\mu_\beta}$. This is done in Section \ref{Sec:MLA}, where we first repeat the main steps for the construction of a Mittag-Leffler analysis from \cite{GJRS14} and we introduce the test function space $(\cN)^1$ and the distribution space $(\cN)^{-1}_{\mu_\beta}$. In this way we give an approximation of Donsker's delta by a sequence of Bochner integrals in the Hilbert space $L^2(\mu_\beta)$.

In Section \ref{Sec:gna}, we demonstrate that generalized grey Brownian motion (short ggBm) $B^{\alpha,\beta}$, $0<\alpha<2$, can be defined in the setting of Mittag-Leffler analysis using the probability space $(\cS'(\R),\cB,\mu_\beta)$, where $\cS'(\R)$ are the tempered distributions, $\cB$ denotes the cylinder $\sigma-$algebra and $\mu_\beta$, $0<\beta<1$, are the Mittag-Leffler measures. This stochastic process was analysed for example in \cite{Sch90, MM09} in a slightly different setting. In Section \ref{Sec:Heateq} we establish the connection of ggBm to the fractional heat equation. This equation has already been studied for example in \cite{SW89, Mai95, Koch90, EK04}. In \cite{OB09} the solution to the time-fractional heat equation is interpreted as the density of the composition of certain stochastic processes. We obtain a time-fractional heat kernel by extending the fractional Feynman-Kac formula \cite{Sch92} to localized (Dirac delta) initial values and justify in this way that ggBm plays the same role for the time-fractional heat equation as Brownian motion does for the usual heat equation. In Section \ref{Sec:gnp} we apply the characterizations from Mittag-Leffler analysis to show differentiability of $B^{\alpha,\beta}$ in $(\cN)^{-1}_{\mu_\beta}$ and we prove the existence of ggBm local times using Donsker's delta, see also \cite{DSE13}. 

Summarizing, the core results of this article are:
\begin{itemize}
	\item Giving a characterization of convergent sequences $(\Phi_n)_{n\in\N}$ in the distribution space $(\cN)^{-1}_{\mu_\beta}$, see Theorem \ref{charconv}.
	\item Providing an approximation of Donsker's delta by a sequence of Bochner integrals, see Theorem \ref{Theo:Approxdonsker}.
	\item Constructing the Green's function to time-fractional heat equations by establishing a time-fractional Feynman-Kac formula, see Theorem \ref{Theo:FFK} and Remark \ref{Rem:FFK}.
	\item Showing differentiability of generalized grey Brownian motion in a distributional sense, see Theorem \ref{Theo:Diff}.
	\item Constructing local times of grey Brownian motions, see Theorem \ref{Theo:LocalTimes}.
\end{itemize}

\section{Mittag-Leffler Analysis}\label{Sec:MLA}

\subsection{Prerequisites}
In this section, we repeat the construction of Mittag-Leffler measures as probability measures on a conuclear space $\cN'$ from \cite{Sch92}. First we need to collect some facts about nuclear triples used in this paper. For details see e.g.~\cite{Sch71, RS}.

Let $\cH$ be a real separable Hilbert space with inner product $(\cdot,\cdot)$ and corresponding norm $\abs{\cdot}$. Let $\cN$ be a nuclear space which is continuously and densely embedded in $\cH$ and let $\cN'$ be its dual space. The canonical dual pairing between $\cN'$ and $\cN$ is denoted by $\lab\cdot,\cdot\rab$, and by identifying $\cH$ with its dual space via the Riesz isomorphism we get the inclusions $\cN\subset\cH\subset\cN'$. In particular $\lab f,\xi\rab=(f,\xi)$ for $f\in\cH$, $\xi\in\cN$.

We assume that $\cN$ can be represented by a countable family of Hilbert spaces as follows: For each $p\in\N$ let $\cH_p$ be a real separable Hilbert space with norm $\abs{\cdot}_p$ such that $\cN\subset\cH_{p+1}\subset\cH_p\subset\cH$ continuously and the inclusion $\cH_{p+1}\subset\cH_p$ is a Hilbert Schmidt operator. It is no loss of generality to assume $\abs{\cdot}_p\le\abs{\cdot}_{p+1}$ on $\cH_{p+1}$ and $\cH_0=\cH$, $\abs{\cdot}_0=\abs{\cdot}$. The space $\cN$ is assumed to be the projective limit of the spaces $(\cH_p)_{p\in\N}$, that is $\cN=\bigcap_{p\in\N}\cH_p$ and the topology on $\cN$ is the coarsest locally convex topology such that all inclusions $\cN\subset\cH_p$ are continuous.

This also gives a representation of $\cN'$ in terms of an inductive limit: Let $\cH_{-p}$ be the dual space of $\cH_p$ with respect to $\cH$ and let the dual pairing between $\cH_{-p}$ and $\cH_p$ be denoted by $\lab\cdot,\cdot\rab$ as well. Then $\cH_{-p}$ is a Hilbert space and we denote its norm by $\abs{\cdot}_{-p}$. It follows by general duality theory that $\cN'=\bigcup_{p\in\N}\cH_{-p}$, and we may equip $\cN'$ with the inductive topology, that is the finest locally convex topology such that all inclusions $\cH_{-p}\subset\cN'$ are continuous. We end up with the following chain of dense and continuous inclusions:
\[ \cN\subset\cH_{p+1}\subset\cH_p\subset\cH\subset\cH_{-p}\subset\cH_{-(p+1)}\subset\cN'. \]

We also use tensor products and complexifications of these spaces. In the following we always identify $f = [f_1,f_2]\in\cH_{p,\C}, f_1,f_2\in\cH_p$ for $p\in\Z$ with $f=f_1 + if_2$. The notation $\abs{\cdot}_p$ is kept for the norm and $\lab\cdot,\cdot\rab$ denotes the bilinear dual pairing.

\begin{remark}\label{Rem:Nprop}
$\cN$ is a perfect space, i.e.~every bounded and closed set in $\cN$ is compact. As a consequence strong and weak convergence coincide in both $\cN$ and $\cN'$, see page 73 in \cite{GV64} and Section I.6.3 and I.6.4 in \cite{GS68}.
\end{remark}

\begin{example}\label{ex:wnsetting}
Consider the white noise setting, where $\cN = \cS(\R)$ is the space of Schwartz test functions, $\cH = L^2(\R,\rmd x)$ and $\cN' = \cS'(\R)$ are the tempered distributions. $\cS(\R)$ is dense in $L^2(\R,\rmd x)$ and can be represented as the projective limit of certain Hilbert spaces $\cH_p$, $p>0$, with norms denoted by $\abs{\cdot}_p$, see e.g. \cite{Kuo}. Thus the white noise setting is an example for the nuclear triple described above.
\end{example}

\subsection{The Mittag-Leffler measure}
As Mittag-Leffler measures $\mu_\beta$, $0<\beta<1$, we denote the probability measures on $\cN'$ whose characteristic functions are given via Mittag-Leffler functions. The Mittag-Leffler function was introduced by G\"osta Mittag-Leffler in \cite{ML05} and we also consider a generalization first appeared in \cite{W05a}. 
\begin{definition}
For $0<\beta<\infty$ the Mittag-Leffler function is an entire function defined by its power series
\[
	\rmE_\beta (z) := \sum_{n=0}^\infty \frac{z^n}{\Gamma(\beta n+1)},\quad z\in\C.
\]
Here $\Gamma$ denotes the well-known Gamma function which is an extension of the factorial to complex numbers such that $\Gamma(n+1)=n!$ for $n\in\N$. Furthermore we define for $\gamma\in\C$ the following entire function of Mittag-Leffler type, see also \cite{EMOT55},
\[
	\rmE_{\beta,\gamma}(z) := \sum_{n=0}^\infty \frac{z^n}{\Gamma(\beta n+\gamma)},\quad z\in\C.
\]
\end{definition}
The Mittag-Leffler function is an entire function and thus absolutely convergent on compact sets. Hence we may interchange sum and derivative and calculate the derivative of $\rmE_\beta$:
\begin{lemma}\label{derEbeta}
For the derivative of the Mittag-Leffler function $\rmE_\beta$, $\beta>0$, it holds
\[
	\frac{\rmd}{\rmd z} \rmE_\beta(z) = \frac{\rmE_{\beta,\beta}(z)}{\beta},\quad z\in\C.
\]
\end{lemma}

The mapping $\lcb t\in\R\mid t>0\rcb\ni t\mapsto\rmE_\beta(-t)\in\R$ is completely monotonic for $\beta\in (0,1]$, see \cite{P48, Fel71}. This is sufficient to show that
\[ \cN\ni\xi\mapsto\rmE_\beta\lb-\halb\lab\xi,\xi\rab\rb\in\R \]
is a characteristic function on $\cN$, see \cite{Sch92}. Using the theorem of Bochner and Minlos, see e.g. \cite{BK95}, the following probability measures on $\cN'$, equipped with its cylindrical $\sigma$-algebra, can be defined:

\begin{definition}\label{Def:MLmeas}
For $\beta\in(0,1]$ the Mittag-Leffler measure is defined to be the unique probability measure $\mu_\beta$ on $\cN'$ such that for all $\xi\in\cN$
\[
	\int_{\cN'} \exp(i\lab\omega,\xi\rab)\,\rmd\mu_\beta(\omega) = \rmE_\beta\lb-\frac{1}{2}\lab\xi,\xi\rab\rb.
\]
The corresponding $L^p$ spaces of complex-valued functions are denoted by $L^p(\mu_\beta):=L^p(\cN',\mu_\beta;\C)$ for $p\geq 1$ with corresponding norms $\lVb\cdot\rVb_{L^p(\mu_\beta)}$. By $\E_{\mu_\beta}(f) := \int_{\cN'} f(\omega)\,\rmd\mu_\beta(\omega)$ we define the expectation of $f\in L^1(\mu_\beta)$.
\end{definition}

In \cite{Sch92} all moments of $\mu_\beta$ are calculated:
\begin{lemma}\label{moments2}
Let $\xi\in\cN$ and $n\in\N$. Then
\[ \int_{\cN'} \lab\omega,\xi\rab^{2n+1}\,\rmd\mu_\beta(\omega)=0 \quad \text{and} \quad \int_{\cN'} \lab\omega,\xi\rab^{2n}\,\rmd\mu_\beta(\omega)= \frac{(2n)!}{\Gamma(\beta n+1)2^n}\lab\xi,\xi\rab^n. \]
In particular $\lVb\lab\cdot,\xi\rab\rVb_{L^2(\mu_\beta)}^2=\frac{1}{\Gamma(\beta+1)}\abs{\xi}^2$.
\end{lemma}

\begin{remark}\label{remark:extension}
Using Lemma \ref{moments2} it is possible to define $\lab\cdot,\eta\rab$ for $\eta\in\cH$ as the $L^2(\mu_\beta)$-limit of $\lab\cdot,\xi_n\rab$, where $\lb\xi_n\rb_{n\in\N}$ is a sequence in $\cN$ converging to $\eta$ in $\cH$.
\end{remark}

\begin{lemma}\label{Lemma:isometry}
For $\xi,\eta\in\cN$ it holds that
\[
	\int_{\cN'} \lab\omega,\xi\rab\lab\omega,\eta\rab\,\rmd\mu_\beta = \frac{1}{\Gamma(\beta+1)} \lab\xi,\eta\rab.
\]
\end{lemma}

\begin{proof}
By definition we have for each $a_1,a_2\in\R$
\[
	\int_{\cN'} \exp\lb i\lab\omega,a_1\xi+a_2\eta\rab\rb\,\rmd\mu_\beta = \sum_{n=0}^\infty \frac{(-1)^n}{2^n\Gamma(\beta n+1)}  \lab a_1\xi+a_2\eta,a_1\xi+a_2\eta\rab^{n} .
\]
Applying the operator $-\tfrac{\partial^2}{\partial a_1\partial a_2}$ to both sides and evaluate at $a_1=a_2=0$ we get the result. Interchanging of sum and derivative is possible since the Mittag-Leffler function is entire. Further for all $a_1,a_2\in (-1,1)$ it holds
\begin{align*}
	\abs{\frac{\partial}{\partial a_1} i\lab \omega,\eta\rab\exp\lb i\lab\omega,a_1\xi+a_2\eta\rab\rb} &\leq \abs{\lab\omega,\eta\rab} \e^{\abs{a_2}{\lab\omega,\eta\rab}} \abs{\lab\omega,\xi\rab} \e^{\abs{a_1}\abs{\lab\omega,\xi\rab}} \\
	&\leq \e^{2\abs{\lab\omega,\xi\rab}} \e^{2\abs{\lab\omega,\eta\rab}} := g(\omega).
\end{align*}
$g\in L^1(\mu_\beta)$ by Lemma 4.1 in \cite{GJRS14}. Hence $\partial/\partial a_1$ interchanges with the integral. A similar argument yields that also the integral and $\partial/\partial a_2$ interchange.
\end{proof}

\begin{remark}
It is shown in \cite{GJRS14} that $f\in L^2(\mu_\beta)$ does not admit a chaos expansion if $\beta\neq 1$. This means that there is no system of polynomials $\lb I_n(\xi)\rb_{n\in\N}$, $\xi\in\cN$, on $\cN'$ fulfilling the following properties simultaneously:
\begin{enumerate}[label=\textit{(\roman*)}]
	\item $I_n(\xi) = p_{n,\xi}(\lab\cdot,\xi\rab)$ for some monic polynomial $p_{n,\xi}$ on $\R$ or $\C$ of degree $n\in\N$.
	\item $I_n(\xi) \perp I_m(\xi)$ in $L^2(\mu_\beta)$ for $n\neq m$.
	\item $I_n(\xi) \perp I_n(\eta)$ in $L^2(\mu_\beta)$ if $\xi\perp\eta$ in $\cH$.
\end{enumerate}
This is in contrast to the case $\beta=1$ where the Mittag-Leffler measure is a Gaussian measure. It is a well-known fact in Gaussian analysis that such a system of orthogonal polynomials, the Wick-ordered polynomials, always exists, see e.g. \cite{Kuo, HKPS}. 
\end{remark}

\subsection{Distributions and Donsker's delta}
Instead of using a system of orthogonal polynomials, it was proposed in \cite{GJRS14} to use Appell systems, compare \cite{KSWY98}. These are biorthogonal systems allowing to construct a test function and a distribution space. As shown in \cite{GJRS14} the measures $\mu_\beta$, $0<\beta<1$, satisfy the following:
\begin{enumerate}\label{Assumptions}
	\item[(A1)] The measure $\mu_\beta$ has an analytic Laplace transform in a neighborhood of zero, i.e. the mapping
	\[
		\cN_\C\ni\theta\mapsto l_{\mu_\beta}(\varphi) := \int_{\cN'} \exp(\lab\omega,\theta\rab) \, \rmd\mu_\beta(\omega) = \rmE_\beta\lb\halb\lab\theta,\theta\rab\rb\in\C
	\]
	is holomorphic in a neighborhood $\cU\subset\cN_\C$ of zero.
	\item[(A2)] For any nonempty open subset $\cU\subset\cN_\C'$ it holds that $\mu_\beta(\cU)> 0$.
\end{enumerate}

We introduce the space of smooth polynomials on $\cN'$, denoted by $\cP(\cN')$ and consisting of finite linear combinations of functions of the form $\lab\cdot,\xi\rab^n$, where $\xi\in\cN_\C$ and $n\in\N$. Every smooth polynomial $\varphi$ has a representation
\[ \varphi(\omega)=\sum_{n=0}^N\lab\omega^\otn,\varphi^{(n)}\rab, \]
where $N\in\N$ and $\varphi^{(n)}$ is a finite sum of elements of the form $\xi^\otn$, $\xi\in\cN_\C$. We equip $\cP(\cN')$ with the natural topology such that the mapping
\[
	\varphi = \sum_{n=0}^\infty \lab \cdot^\otn,\varphi^{(n)}\rab \leftrightarrow \vec{\varphi}=\lcb \varphi^{(n)} : n\in\N \rcb
\]
becomes a topological isomorphism from $\cP(\cN')$ to the topological direct sum of tensor powers $\cN^\wotn_\C$, i.e.
\[
	\cP(\cN') \simeq \bigoplus_{n=0}^\infty \cN^\wotn_\C
\]
(note that $\varphi^{(n)}\neq 0$ only for finitely many $n\in\N$).
Then we introduce the space $\cP'_{\mu_\beta}(\cN')$ as the dual space of $\cP(\cN')$ with respect to $L^2(\mu_\beta)$, i.e.
\[
	\cP(\cN') \subset L^2(\mu_\beta) \subset \cP'_{\mu_\beta}(\cN')
\]
and the dual pairing $\ddp{\cdot}{\cdot}_{\mu_\beta}$ between $\cP'_{\mu_\beta}(\cN')$ and $\cP(\cN')$ is a bilinear extension of the scalar product on $L^2(\mu_\beta)$ by
\[
	\ddp{f}{\varphi}_{\mu_\beta} = (f,\overline{\varphi})_{L^2(\mu_\beta)},\quad\varphi\in\cP(\cN'),f\in L^2(\mu_\beta).
\]
Note that (A1) ensures that $\cP(\cN')\subset L^2(\mu_\beta)$ is dense, see \cite{Sko74}. Then it is possible to construct an Appell System $\A^{\mu_\beta} = (\P^{\mu_\beta},\Q^{\mu_\beta})$ generated by the measure $\mu_\beta$. We will only give the main definitions, for further details and proofs we refer to \cite{GJRS14, KSWY98}. 

For each $n\in\N$ and $z\in\cN'_\C$ we construct $P_n^{\mu_\beta}(z)\in\lb\cN^{\hat{\otimes}n}_\C\rb'$ such that the Appell polynomials
\[
	\P^{\mu_\beta} = \lcb \lab P^{\mu_\beta}_n(\cdot),\varphi^{(n)} \rab \; \Big| \; \varphi^{(n)}\in\cN^{\hat{\otimes} n}_\C,\;n\in\N\rcb,
\]
give a representation for $\varphi\in\cP(\cN')$ by
\begin{equation}\label{chaosdecompI}
	\varphi = \sum_{n=0}^N \lab P^{\mu_\beta}_n(\cdot),\varphi^{(n)} \rab 
\end{equation}
for suitable $N\in\N$ and $\varphi^{(n)}\in\cN^{\hat{\otimes} n}_\C$. Every distribution $\Phi\in\cP'_{\mu_\beta}(\cN')$ is of the form 
\begin{equation}\label{chaosdecompII}
	\Phi=\sum_{n=0}^\infty Q^{\mu_\beta}_n(\Phi^{(n)}),
\end{equation}
where $Q^{\mu_\beta}_n(\Phi^{(n)})$ is from the $\Q^{\mu_\beta}$-system
\[
	\Q^{\mu_\beta} = \lcb Q^{\mu_\beta}_n(\Phi^{(n)}) \; \Big| \; \Phi^{(n)}\in\lb\cN^{\hat{\otimes}n}_\C\rb',\;n\in\N\rcb.
\]
Furthermore the dual pairing of a distribution and a test function is given by 
\begin{equation}\label{eq:biorth}
	\ddp{Q^{\mu_\beta}_n\lb\Phi^{(n)}\rb}{\lab P^{\mu_\beta}_m,\varphi^{(m)}\rab}_{\mu_\beta} = \delta_{m,n} n! \lab \Phi^{(n)},\varphi^{(n)}\rab,\quad n,m\in\N_0,
\end{equation}
for $\Phi^{(n)}\in\lb\cN^{\hat{\otimes}n}_\C\rb'$ and $\varphi^{(m)}\in\cN^{\hat{\otimes}m}_\C$. With the help of the Appell system $\A^{\mu_\beta}$ a test function and a distribution space can now be constructed, see \cite{KSWY98}, which results in the following chain of spaces
\[
	(\cN)^1_{\mu_\beta} \subset (\cH_p)^1_{q,\mu_\beta} \subset L^2(\mu_\beta) \subset (\cH_{-p})^{-1}_{-q,\mu_\beta} \subset (\cN)^{-1}_{\mu_\beta}\quad p,q\in\N.
\]
Here, $\lb\cH_p\rb^1_{q,\mu_\beta}$ denotes the completion of $\cP(\cN')$ with respect to $\norm{\cdot}_{p,q,\mu_\beta}$, given by
\[
	\norm{\varphi}^2_{p,q,\mu_\beta} := \sum_{n=0}^N (n!)^2 2^{nq}\abs{\varphi^{(n)}}_p^2,\quad p,q\in\N_0, \varphi\in\cP(\cN').
\]
There are $p',q'>0$ such that $\lb\cH_p\rb^1_{q,\mu_\beta}$ is topologically embedded in $L^2(\mu_\beta)$ for all $p>p',q>q'$, see \cite{KK99}. Moreover, the set of $\mu_\beta$-exponentials $\lcb e_{\mu_\beta}(\theta;\cdot) := \tfrac{\e^{\lab\cdot,\theta\rab}}{l_{\mu_\beta}(\theta)} \; \big| \; \theta\in U_{p,q} \rcb$ is total in $(\cH_p)_{q,\mu_\beta}^1$. Here the neighborhood $U_{p,q}\subset \cN_\C$ of zero is defined by $U_{p,q} := \lcb \theta\in\cN_\C \mid \abs{\theta}_p < 2^{-q}\rcb$. Since the $\mu_\beta$-exponentials have the series expansion $e_{\mu_\beta}(\theta,\cdot) = \sum_{n=0}^\infty 1/(n!)\lab P^{\mu_\beta}_n(\cdot),\theta^\otn\rab$ it holds that for each $\theta\in U_{p,2q}$ the norm is given by
\begin{equation}\label{eq:muexpNorm}
	\norm{e_{\mu_\beta}(\theta,\cdot)}_{p,q,\mu_\beta}^2 = \sum_{n=0}^\infty 2^{-nq} = (1-2^{-q})^{-1}.
\end{equation}

By $(\cH_{-p})^{-1}_{-q,\mu_\beta}$ we denote the set of all $\Phi\in\cP'_{\mu_\beta}(\cN')$ for which $\norm{\Phi}_{-p,-q,\mu_\beta}$ is finite, where 
\[
	\norm{\Phi}^2_{-p,-q,\mu} := \sum_{n=0}^\infty 2^{-qn}\abs{\Phi^{(n)}}^2_{-p},\quad p,q\in\N_0.
\]
It holds that $(\cH_{-p})^{-1}_{-q,\mu_\beta}$ is the dual of $(\cH_p)^1_{q,\mu_\beta}$ and the test function space $\lb\cN\rb^1_{\mu_\beta}$ is defined as the projective limit of $\lb\cH_p\rb^1_{q,\mu_\beta}$. This is a nuclear space which is continuously embedded in $L^2(\mu_\beta)$. Moreover it turns out that the test function space $(\cN)^1_{\mu_\beta}$ is the same for all measures $\mu_\beta$ satisfying (A1) and (A2), thus we will just use the notation $(\cN)^1$. The space of distributions $(\cN)_{\mu_\beta}^{-1}$ is the inductive limit of $(\cH_{-p})^{-1}_{-q,\mu_\beta}$ and is the dual of $(\cN)^1$ with respect to $L^2(\mu_\beta)$ and the dual pairing between a distribution $\Phi\in(\cN)_{\mu_\beta}^{-1}$ as in \eqref{chaosdecompII} with a test function $\varphi\in(\cN)^1$ as in \eqref{chaosdecompI} is given by 
	\[
		\ddp{\Phi}{\varphi}_{\mu_\beta} = \sum_{n=0}^\infty n! \lab\Phi^{(n)},\varphi^{(n)}\rab.
	\]
We shall use the same notation for the dual pairing between $(\cH_{-p})^{-1}_{-q,\mu_\beta}$ and $(\cH_p)^1_{q,\mu_\beta}$.

As in \cite{GJRS14} we introduce the $S_{\mu_\beta}$-transform of $\Phi\in (\cN)^{-1}_{\mu_\beta}$ on a suitable neighborhood $\cU\subset\cN_\C$ of zero by
\[
	S_{\mu_\beta}\Phi(\theta) = \ddp{\Phi}{e_{\mu_\beta}(\theta,\cdot)}_{\mu_\beta} = \frac{1}{\rmE_\beta\lb\halb\lab\theta,\theta\rab\rb} \ddp{\Phi}{\e^{\lab\cdot,\theta\rab}}_{\mu_\beta}, \quad\theta\in\cU.
\]
Furthermore we define the $T_{\mu_\beta}$-transform of $\Phi\in (\cN)^{-1}_{\mu_\beta}$ at $\theta\in\cU$ by
\[
	T_{\mu_\beta}\Phi(\theta) = \ddp{\Phi}{\exp(i\lab\cdot,\theta\rab}_{\mu_\beta}.
\]
For $\Phi\in(\cN)^{-1}_{\mu_\beta}$ as in \eqref{chaosdecompII} we have
\begin{equation}\label{stransform}
	S_{\mu_\beta}\Phi(\theta) = \sum_{n=0}^\infty \lab\Phi^{(n)},\theta^{\otimes n}\rab, \quad\theta\in\cU.
\end{equation}
The relation between $T_{\mu_\beta}$- and $S_{\mu_\beta}$-transform is given by, see \cite{GJRS14}
\begin{equation}\label{s-t-transform}
	T_{\mu_\beta}\Phi(\theta) = l_{\mu_\beta}(i\theta) S_{\mu_\beta} \Phi(i\theta), \quad\theta\in\cU.
\end{equation}

The space $(\cN)^{-1}_{\mu_\beta}$ can be characterized via the $S_{\mu_\beta}$-transform using spaces of holomorphic functions on $\cN_\C$. By $\Hol$ we denote the space of all holomorphic functions at zero. Let $F$ and $G$ be holomorphic on a neighborhood $\cV,\cU\subset\cN_\C$ of zero, respectively. We identify $F$ and $G$ if there is a neighborhood $\cW\subset\cV$ and $\cW\subset\cU$ such that $F(\theta) = G(\theta)$ for all $\theta\in\cW$. $\Hol$ is the union of the spaces
\[
	\lcb F\in\Hol \;\Big|\; n_{p,l,\infty}(F) = \sup_{\abs{\theta}_p\leq2^{-l}} \abs{F(\theta)} <\infty\rcb,\quad p,l\in\N
\]
and carries the inductive limit topology.

The following theorem is proven in \cite{KSWY98}:
\begin{theorem}\label{characterization}
The $S_{\mu_\beta}$-transform is a topological isomorphism from $(\cN)^{-1}_{\mu_\beta}$ to $\Hol$. Moreover, if $F\in\Hol,\;F(\theta) = \sum_{n=0}^\infty \lab\Phi^{(n)},\theta^\otn\rab$ for all $\theta\in\cN_\C$ with $\abs{\theta}_p\leq 2^{-l}$ and if $p'>p$ with $\norm{i_{p',p}}_{HS}<\infty$ and $q\in\N$ such that $\rho := 2^{2l-q}\e^2\norm{i_{p',p}}_{HS}^2<1$ then $\Phi = \sum_{n=0}^\infty Q^{\mu_\beta}_n(\Phi^{(n)})\in\lb\cH_{-p'}\rb^{-1}_{-q}$ and
\[
	\norm{\Phi}_{-p',q,{\mu_\beta}} \leq n_{p,l,\infty}(F) (1-\rho)^{-1/2}.
\]
\end{theorem}

As a corollary from the characterization theorem \cite{GJRS14} proves a result which describes the integrable mappings (in a weak sense) with values in $(\cN)^{-1}_{\mu_\beta}$: 

\begin{theorem}\label{charint}
Let $(T,\cB,\nu)$ be a measure space and $\Phi_t\in (\cN)^{-1}_{\mu_\beta}$ for all $t\in T$. Let $\cU\subset\cN_\C$ be an appropriate neighborhood of zero and $C<\infty$ such that:
\begin{enumerate}[label=\textit{(\roman*)}]
	\item $S_{\mu_\beta}\Phi_\cdot(\theta)\colon T\to\C$ is measurable for all $\theta\in\cU$.
	\item $\int_T \abs{S_{\mu_\beta}\Phi_t(\theta)}\,\rmd\nu(t) \leq C$ for all $\theta\in\cU$.
\end{enumerate}
Then there exists $\Psi\in (\cN)^{-1}_{\mu_\beta}$ such that for all $\theta\in\cU$
\[
	S_{\mu_\beta}\Psi(\theta) = \int_T S_{\mu_\beta}\Phi_t(\theta)\,\rmd\nu(t).
\]
We denote $\Psi$ by $\int_T \Phi_t\,\rmd\nu(t)$ and call it the weak integral of $\Phi$.
\end{theorem}

Another consequence of Theorem \ref{characterization} characterizes the convergent sequences in $(\cN)^{-1}_{\mu_\beta}$:
\begin{theorem}\label{charconv}
Let $(\Phi_n)_{n\in\N}$ be a sequence in $(\cN)^{-1}_{\mu_\beta}$. Then $(\Phi_n)_{n\in\N}$ converges strongly in $(\cN)^{-1}_{\mu_\beta}$ if and only if there exist $p,q\in\N$ with the following two properties:
\begin{enumerate}[label=\textit{(\roman*)}]
\item $\lb S_{\mu_\beta}\Phi_n(\theta)\rb_{n\in\N}$ is a Cauchy sequence for all $\theta\in U_{p,q}$.
\item $S_{\mu_\beta}\Phi_n$ is holomorphic on $U_{p,q}$ and there is a constant $C>0$ such that \[\abs{S_{\mu_\beta}\Phi_n(\theta)}\leq C\] for all $\theta\in U_{p,q}$ and for all $n\in\N$.
\end{enumerate}
\end{theorem}

\begin{proof}
First assume that $(i)$ and $(ii)$ hold. From \eqref{stransform} and by Theorem \ref{characterization} it follows that there exist $p',q'\in\N$ such that
\begin{align}\label{eq:normbound}
	\norm{\Phi_n}_{-p',-q',{\mu_\beta}} &\leq n_{p,q,\infty}(S_{\mu_\beta}\Phi_n)(1-\rho)^{-1/2} \leq C(1-\rho)^{-1/2}.
\end{align}
Because of \textit{(i)} together with \eqref{eq:normbound} and since the ${\mu_\beta}$-exponentials are a total set in $(\cH_{p'})_{q'}^1$ we conclude that the sequence $\lb\ddp{\Phi_n}{\varphi}_{\mu_\beta}\rb_{n\in\N}$ is a Cauchy sequence for all $\varphi\in(\cH_{p'})_{q'}^1$ by approximating $\varphi$ with $\mu_\beta$-exponentials. Indeed, for $\varepsilon>0$ choose $e$ from the linear span of the $\mu_\beta$-exponentials such that $\norm{\varphi-e}_{p',q',\mu_\beta} < 1/(4C)(1-\rho)^{1/2} \varepsilon$ and choose $m,n$ large enough such that $\abs{\ddp{\Phi_m-\Phi_n}{e}} <\varepsilon/2$. Then
\[
	\abs{\ddp{\Phi_n-\Phi_m}{\varphi}_{\mu_\beta}} \leq \abs{\ddp{\Phi_n-\Phi_m}{e}} + \norm{\Phi_n+\Phi_m}_{-p',-q',\mu_\beta} \norm{\varphi-e}_{p',q',\mu_\beta} <\varepsilon.
\]
Thus the mapping $\Phi\colon (\cH_{p'})^1_{q',\mu_\beta} \to \C$, $\Phi(\varphi) := \lim_{n\to\infty} \ddp{\Phi_n}{\varphi}$, $\varphi\in(\cH_{p'})_{q'}^1$ is well-defined, linear and continuous since 
\[
	\abs{\Phi(\varphi)} \leq \liminf_{n\to\infty} \norm{\Phi_n}_{-p',-q',\mu_\beta} \norm{\varphi}_{p',q',\mu_\beta} \leq C(1-\rho)^{-1/2} \norm{\varphi}_{p',q',\mu_\beta}.
\] 
Hence $\Phi\in (\cH_{-p'})^{-1}_{-q',\mu_\beta}$. This shows that $(\Phi_n)_{n\in\N}$ converges weakly to $\Phi$ in $(\cN)^{-1}_{\mu_\beta}$ since $(\cN)^1$ is reflexive. Finally we use that in the dual space of a nuclear space strong and weak convergence coincide, see Remark \ref{Rem:Nprop}. Hence $\lb\Phi_n\rb_{n\in\N}$ converges strongly to $\Phi$ in $(\cN)^{-1}_{\mu_\beta}$.

Conversely let $\lb\Phi_n\rb_{n\in\N}$ converge strongly to some $\Phi\in(\cN)^{-1}_{\mu_\beta}$. Strong convergence implies that there exist $p,q\in\N$ such that $\Phi_n\to\Phi$ in $(\cH_{-p})^{-1}_{-q,\mu_\beta}$ as $n\to\infty$. Thus $(i)$ is obviously fulfilled for all $\theta\in U_{p,q}$. The strong convergence also implies that there exists $K<\infty$ such that $\sup_{n\in\N} \norm{\Phi_n}_{-p,-q,\mu_\beta} <K$. Thus we have for all $\theta\in U_{p,2q}$ and for all $n\in\N$ that 
\[
	\abs{(S_{\mu_\beta}\Phi_n)(\theta)} \leq \norm{\Phi_n}_{-p,-q,\mu_\beta} \norm{e_{\mu_\beta}(\theta;\cdot)}_{p,q,\mu_\beta} \leq K(1-2^{-q})^{-1/2} <\infty,
\]
see \eqref{eq:muexpNorm}. This shows $(ii)$.
\end{proof}

\begin{remark}
Because of \eqref{s-t-transform} $T_{\mu_\beta}\Phi$ is holomorphic if and only $S_{\mu_\beta}\Phi$ is holomorphic. Thus the characterization theorems \ref{characterization}, \ref{charint} and \ref{charconv} also hold if the $S_{\mu_\beta}$-transform is replaced by the $T_{\mu_\beta}$-transform.
\end{remark}

The characterization theorems have already successfully been applied in Mittag-Leffler analysis. A generalization of Donsker's delta from Gaussian analysis could be constructed in $(\cN)^{-1}_{\mu_\beta}$ via Theorem \ref{charint} and its $T_{\mu_\beta}$-transform is calculated, see \cite{GJRS14}. In fact, for $\eta\in\cN$ and $a\in\R$ Donsker's delta is defined as a weak integral in the sense of Theorem \ref{charint} by
\begin{equation}\label{eq:DonskersDelta}
	\delta_a(\lab\cdot,\eta\rab) = \frac{1}{2\pi}\int_\R \exp\lb ix(\lab\cdot,\eta\rab-a) \rb\,\rmd x \in \lb\cN\rb^{-1}_{\mu_\beta}.
\end{equation}
We prove in the following theorem that the weak integral in \eqref{eq:DonskersDelta} can be approximate by a sequence of Bochner integrals from $L^2(\mu_\beta)$. Thus Donsker's delta can be represented as a limit of square integrable functions.

\begin{theorem}\label{Theo:Approxdonsker}
For $\eta\in\cH$ and $a\in\R$ it holds that
\[
	\delta_a(\lab\cdot,\eta\rab) = \lim_{n\to\infty} \frac{1}{2\pi} \int_{-n}^n \e^{ix(\lab\cdot,\eta\rab-a)}\,\rmd x \text{ in }(\cN)^{-1}_{\mu_\beta}.
\]
\end{theorem}

\begin{proof}
Set $\Phi_n := (2\pi)^{-1} \int_{-n}^n \e^{ix(\lab\cdot,\eta\rab-a)}\,\rmd x$, $n\in\N$. First note that $\Phi_n\in L^2(\mu_\beta)$ as a Bochner integral since
\[
	\int_{-n}^n \norm{\exp\lb ix(\lab\cdot,\eta\rab-a)\rb}_{L^2(\mu_\beta)}\,\rmd x = \int_{-n}^n 1\,\rmd x =2n <\infty.
\] 
We have for the $T_{\mu_\beta}$-transform of $\Phi_n$ that
\[
	\lb T_{\mu_\beta}\Phi_n\rb(\theta) = \frac{1}{2\pi}\int_{-n}^n \e^{-ixa}\rmE_\beta\Bigl(-\halb x^2\lab\eta,\eta\rab - \halb \lab\theta,\theta\rab - x\lab\theta,\eta\rab \Bigr)\,\rmd x,
\]
see Proposition 5.6 in \cite{GJRS14}. Now it holds that the integrand 
\[
	\1_{[-n,n]}(x)\e^{-ixa}\rmE_\beta\Bigl(-\halb x^2\lab\eta,\eta\rab - \halb \lab\theta,\theta\rab - x\lab\theta,\eta\rab \Bigr)
\] 
converges pointwisely for each $x\in\R$ to $$\e^{-ixa}\rmE_\beta\Bigl(-\halb x^2\lab\eta,\eta\rab - \halb \lab\theta,\theta\rab - x\lab\theta,\eta\rab \Bigr)$$ as $n\to\infty$ and it is bounded by $\abs{\rmE_\beta\Bigl(-\halb x^2\lab\eta,\eta\rab - \halb \lab\theta,\theta\rab - x\lab\theta,\eta\rab \Bigr)}$. It was shown in Proposition 5.2 in \cite{GJRS14} that there is a neighborhood $\cU\subset\cN_\C$ of zero and a constant $C>0$ such that
\[
	\int_\R \abs{\rmE_\beta\Bigl(-\halb x^2\lab\eta,\eta\rab - \halb \lab\theta,\theta\rab - x\lab\theta,\eta\rab \Bigr)}\,\rmd x< C,\quad \theta\in\cU.
\]
Applying dominated convergence we see that $\lb T_{\mu_\beta}\Phi_n\rb(\theta)$ converges as $n\to\infty$ to 
\[
	\frac{1}{2\pi}\int_\R \e^{-ixa}\rmE_\beta\Bigl(-\halb x^2\lab\eta,\eta\rab - \halb \lab\theta,\theta\rab - x\lab\theta,\eta\rab \Bigr)\,\rmd x.
\] 
Moreover for all $\theta\in\cU$ it holds
\[
	\abs{\lb T_{\mu_\beta}\Phi_n\rb(\theta)} \leq \frac{1}{2\pi }\int_\R \abs{\rmE_\beta\Bigl(-\halb x^2\lab\eta,\eta\rab - \halb \lab\theta,\theta\rab - x\lab\theta,\eta\rab \Bigr)}\,\rmd x<C.
\]
Now the assertion follows by applying Theorem \ref{charconv}.
\end{proof}

\section{Grey Noise Analysis}\label{Sec:gna}

\subsection{Basic Definitions}
The main ideas of grey noise analysis go back to Schneider in \cite{Sch90}. He constructed grey Brownian motion on a concrete probability space. Further details were given amongst others by Mura and Mainardi in \cite{MM09} and also in \cite{Kuo, KS93}. We now want to point out that grey noise analysis is a special case of Mittag-Leffler analysis, where the spaces $\cN$ and $\cH$ are chosen in a suitable way as follows.

Consider the space $\cS(\R)$ of Schwartz test functions equipped with the following scalar product
\[
	(\xi,\eta)_\alpha = C(\alpha) \int_\R \abs{x}^{1-\alpha} \overline{\tilde{\xi}(x)}\tilde{\eta}(x) \,\rmd x, \quad \xi,\eta\in\cS(\R).
\]
Here, $0<\alpha<2$ and $C(\alpha) = \Gamma(\alpha+1)\sin(\tfrac{\pi\alpha}{2})$. The notation $\tilde{\eta}$ stands for the Fourier transform of $\eta\in\cS(\R)$, which is defined by
\[
	\tilde{\eta}(x) = (\cF\eta)(x) = \frac{1}{\sqrt{2\pi}}\int_\R \eta(t) \e^{itx}\,\rmd t, \quad x\in\R.
\]
By $\norm{\cdot}_\alpha$ we denote the norm coming from $(\cdot,\cdot)_\alpha$. Define the Hilbert space $\cH_\alpha$ to be the abstract completion of $\cS(\R)$ with respect to $(\cdot,\cdot)_\alpha$. \cite{Sch92} gives an orthonormal system $\lb h^\alpha_n\rb_{n\in\N}$ in $\cS(\R)$ with respect to $(\cdot,\cdot)_\alpha$ and \cite{MM09} constructed an operator $A^\alpha$ on $\cH_\alpha$ such that $A^\alpha h_n^\alpha = (2n+2+1-\alpha)h_n^\alpha$. This allows to define the spaces 
\[
	\cH_{\alpha,p} := \lcb \varphi\in\cH_\alpha \mid \norm{(A^\alpha)^p\varphi}_\alpha <\infty \rcb.
\]

\begin{remark}
Define $\cS_\alpha:=\prlim_{p\to\infty} \cH_{\alpha,p}$ and $\cS_\alpha'$ its dual space. The above considerations show that $\cS_\alpha \subset \cH_\alpha \subset \cS'_\alpha$ is a nuclear triple. Thus we may introduce a Mittag-Leffler measure $\mu_\beta$ on $\cS'_\alpha$. Nevertheless we propose a slightly different choice of the nuclear triple $\cN\subset\cH\subset\cN'$ for two reasons.
\begin{enumerate}
	\item The space $\cH_\alpha$ is defined as an abstract completion. Therefore its elements are Cauchy sequences. For the definition of ggBm later on it is necessary to have $\ind\in\cH_\alpha$. The question arises in which sense the indicator functions $\1_{[a,b)}$, $a,b\in\R$, are in $\cH_\alpha$. A way out would be to identify the indicator functions with a sequence of Schwartz test functions $(\xi_n)_{n\in\N}$, where $(\xi_n)_{n\in\N}$ approximates the indicator function with respect to $\norm{\cdot}_\alpha$. This is still open. Further the later analysis becomes more complicated.
	\item The nature of the test function space $\cS_\alpha$, given as the projective limit of the spaces $\cH_{\alpha,p}$ is not obvious. It is not clear if $\cS_\alpha$ coincides with $\cS(\R)$. 
\end{enumerate}
Therefore we find it useful to develop in the following an alternative approach to a grey noise analysis.
\end{remark}

We consider the usual nuclear triple from white noise analysis, see Example \ref{ex:wnsetting}. The corresponding distribution space $(\cN)^{-1}_{\mu_\beta}$ will be denoted by $(\cS)^{-1}_{\mu_\beta}$.
Define the operator $\Mhpm$ on $\cS(\R)$ by	
\[
	\Mhpm f := \begin{cases}
		K_H D^{-(H-\halb)}_{\pm}f, & H\in (0,\halb), \\
		f, & H=\halb, \\
		K_H I^{H-\halb}_{\pm}f, & H\in (\halb,1), 
	\end{cases} 
\]
with normalisation constant 
\[	
	K_H:=\Gamma\lb H+\halb\rb\lb\int_0^\infty (1+s)^{H-\halb}-s^{H-\halb}\,\rmd s+\frac{1}{2H}\rb^{-\halb}.
\]
\begin{remark}\label{Rem:propMH}
\begin{enumerate}
	\item $D^\alpha_\pm$ and $I^\alpha_\pm$ denote the right- and left-sided fractional derivative and fractional integral of order $\alpha$, respectively. More details are given in Appendix \ref{Sec:Frac}, below. Note that it is not important for the definition of $M^H_\pm$, whether the fractional derivative is in the sense of Riemann-Liouville, Caputo or Marchaud, since all these fractional derivatives coincide on $\cS(\R)$, see Theorem \ref{Theo:EqualityonS}, below. A necessary condition for $f$ to be in the domain of the Caputo derivative is that $f$ is differentiable. Hence the Caputo derivative of indicator functions is not defined. Therefore the operator $M^H_\pm$ denotes in the following the Riemann-Liouville fractional derivative or the Marchaud fractional derivative.  
	\item Although defined on $\cS(\R)$, the domain of $M^H_\pm$ is larger. In particular, $M^H_\pm$ can be applied to indicator functions $\1_{[a,b)}$, $-\infty<a<b<\infty$, and it holds for $t\in\R$ that 
	\[
		\lb M^H_\pm \1_{[a,b)}\rb(t) = \frac{1}{\Gamma(1+H-1/2)} \lb \mp(b-t)^{H-1/2}_\mp \pm (a-t)^{H-1/2}_\mp\rb,
	\]
	see Example \ref{Ex:indicator}, below. Here and below, we use for $\alpha\in\R$ the convention 
\[
	t_+^\alpha := \begin{cases} t^\alpha, & t>0,\\ 0, & t\leq 0, \end{cases} \quad \text{ and } \quad t_-^\alpha := \begin{cases} 0, & t\geq 0, \\ (-t)^\alpha, & t<0. \end{cases}
\] 
	Theorem \ref{Def:fracint} and Corollary \ref{Cor:integrable} below show that 
\[
	M^H_\pm \1_{[a,b)}\in L^2(\R,\rmd x).
\]
\end{enumerate}
\end{remark}

The following result for the Fourier transform of a fractional integral or a fractional derivative of $\varphi\in\cS(\R)$ is known (\cite{SKM}):
\begin{align}
	\cF I_{\pm}^\alpha\varphi(x) = \tilde{\varphi}(x)(\mp ix)^{-\alpha} \label{FTofI} \\
	\cF D_{\pm}^\alpha\varphi(x) = \tilde{\varphi}(x)(\mp ix)^{\alpha}, \label{FTofD}
\end{align}
where $0<\alpha<1$, $(\mp ix)^\alpha = \abs{x}^\alpha \exp(\mp \frac{\alpha\pi i}{2}\sign(x))$ and $\varphi\in\cS(\R)$. With the help of this result we give a relation between the $\alpha$-scalar product and the usual $L^2$-scalar product, compare \cite{LV14}.

\begin{lemma}\label{Lemma:sps}
The $\alpha$-scalar product and the $L^2(\R,\rmd x)$-scalar product are related by the operator $\Mhpm$. In fact for all $\xi,\eta\in\cS(\R)$ it holds
\[
	(\xi,\eta)_\alpha = (M^{\alpha/2}_\pm\xi,M^{\alpha/2}_\pm\eta)_{L^2(\R,\rmd x)}.
\]
\end{lemma}

\begin{proof}
From \cite{Mish} we know that the constant $K_H$ can be calculated to be	$K_H = \sqrt{2H\sin(\pi H)\Gamma(2H)}$ and thus $K^2_{\alpha/2} = C(\alpha)$. Now let $\varphi\in\cS(\R)$. In the case $1<\alpha<2$ it holds that $M^{\alpha/2}_\pm\varphi\in L^2(\R,\rmd x)$, see Theorem \ref{Def:fracint} below, since $\varphi\in\cS(\R)\subset L^p(\R,\rmd x)$ for all $p\geq 1$. In the case $0<\alpha<1$ note that $\tilde{\varphi}(x)(\mp ix)^{(1-\alpha)/2}$ defines a square integrable function. Thus \eqref{FTofD} and Plancherel's theorem imply that $D_{\pm}^{(1-\alpha)/2}\varphi \in L^2(\R,\rmd x)$. Hence, using \eqref{FTofI} and \eqref{FTofD}, we directly see that
\begin{align*}
	\lb \widetilde{M^{\alpha/2}_\pm\xi}, \widetilde{M^{\alpha/2}_\pm\eta} \rb_{L^2(\R,\rmd x)} =  (\xi,\eta)_\alpha. 
\end{align*}
Now the assertion follows by Plancherel's theorem.
\end{proof}

\begin{corollary}
For every $a,b,c,d\in\R$ it holds
\[
	\lb \1_{[a,b)},\1_{[c,d)} \rb_\alpha = \lb M^{\alpha/2}_\pm\1_{[a,b)},M^{\alpha/2}_\pm\1_{[c,d)} \rb_{L^2(\R,\rmd x)}.
\]
\end{corollary}

\begin{proof}
From Remark \ref{Rem:propMH} we already know that $M^{\alpha/2}_\pm\1_{[a,b)}\in L^2(\R,\rmd x)$. For $1<\alpha<2$ it is calculated in \cite{SKM} that for all $\xi\in L^1(\R,\rmd x)$ it holds 
\[
	\cF M^{\alpha/2}_\pm\xi (x) = \cF I_{\pm}^{{(\alpha-1)/2}}\xi(x) = \tilde{\xi}(x)(\mp \ii x)^{(1-\alpha)/2}, \quad x\in\R.
\]
The same holds for $0<\alpha<1$, see Lemma \ref{Lem:fracDerInd}, below. The assertion follows now by the same arguments as in the proof of Lemma \ref{Lemma:sps}.
\end{proof}

\begin{corollary}\label{Cor:covariance}
For all $s,t\geq 0$ and for all $0<\alpha<2$ it holds that
\[
	\lb M^{\alpha/2}_-\1_{[0,t)}, M^{\alpha/2}_-\1_{[0,s)}\rb_{L^2(\R,\rmd x)} = \halb (t^\alpha+s^\alpha-\abs{t-s}^\alpha).
\]
\end{corollary}

\begin{proof}
Use the previous lemma and Example 3.1 in \cite{MM09}.
\end{proof}

As in Definition \ref{Def:MLmeas} we define the corresponding Mittag-Leffler measures $\mu_\beta$, $0<\beta<1$, on $\cS'(\R)$ by
\[
	\int_{\cS'(\R)} \e^{i\lab\omega,\varphi\rab}\,\rmd\mu_\beta(\omega) = \rmE_\beta\lb-\halb\lab\varphi,\varphi\rab\rb, \quad\varphi\in\cS(\R).
\]
The measures $\mu_\beta$ are referred to as \textit{grey noise reference measure}. For $\beta=1$ this is the Gaussian white noise measure. Lemma \ref{moments2} gives all moments of $\mu_\beta$ and Remark \ref{remark:extension} allows the extension of the dual pairing $\lab\cdot,\cdot\rab$ to $\cS'(\R) \times L^2(\R,\rmd x)$. Note that $M^{\alpha/2}_\pm\ind\in L^2(\R,\rmd x)$ for all $t\geq 0$ by Corollary \ref{Cor:integrable} and Theorem \ref{Def:fracint}, below. Thus the following definition makes sense:

\begin{definition}\label{Def:gBm}
For $0<\alpha<2$ we define the process
\[
	\cS'(\R)\ni\omega\mapsto B^{\alpha,\beta}_t(\omega) := \lab\omega, M^{\alpha/2}_- \ind \rab,\quad t>0,
\]
and call this process generalized grey Brownian motion. In the case $\alpha=\beta$ we write $B_t^\alpha$ instead of $B_t^{\alpha,\alpha}$. $B_t^\alpha$ is called grey Brownian motion.
\end{definition}

\begin{remark}
In the approach of \cite{MM09} the grey noise measure is defined via the characteristic function $\rmE_\beta(-(\cdot,\cdot)_\alpha)$ and denoted by $\mu_{\alpha,\beta}$. This means that first the parameters $0<\alpha<2$ and $0<\beta<1$ are fixed and then generalized grey Brownian motion $B_t^{\alpha,\beta}$ is constructed in $L^2(\mu_{\alpha,\beta})$. The measure $\mu_\beta$ as defined above is named grey noise reference measure since for fixed $0<\beta<1$ all generalized grey Brownian motions $B_t^{\alpha,\beta}$ for $0<\alpha<2$ can be constructed in the single space $L^2(\mu_\beta)$. 
\end{remark}

\begin{proposition}
Let $0<\alpha<2$ and $0<\beta<1$. Then for all $p\in\N$ there exists $K<\infty$ such that
\[
	\E_{\mu_\beta}\lb\abs{B_t^{\alpha,\beta} - B_s^{\alpha,\beta}}^{2p}\rb \leq K \abs{t-s}^{\alpha p}, \quad t,s\geq 0.
\]
\end{proposition}

\begin{proof}
Without loss of generality let $s<t$. First note that $\norm{\1_{[s,t)}}_\alpha = \norm{\1_{[0,t-s)}}_\alpha$. 
Corollary \ref{Cor:covariance} shows that $\norm{\1_{[s,t)}}^2_\alpha = (t-s)^\alpha$. Using the moments of $\mu_\beta$, see Lemma \ref{moments2}, we find
\begin{align*}
	\E_{\mu_\beta}\lb \abs{B_t^{\alpha,\beta} - B_s^{\alpha,\beta}}^{2p}\rb &= \frac{(2p)!}{2^p\Gamma(\beta p+1)} \norm{\1_{[s,t)}}^{2p}_\alpha \\
	&= \frac{(2p)!}{2^p\Gamma(\beta p+1)} (t-s)^{\alpha p}.
\end{align*}

\end{proof}

The last proposition ensures that each generalized grey Brownian motion has a continuous version. Indeed, choose $p\in\N$ such that $\alpha p>1$ then the previous proposition provides the estimate  $\E_{\mu_\beta}( (B_t^{\alpha,\beta} - B_s^{\alpha,\beta})^{2p}) \leq K \abs{t-s}^{1+q}$ with $q=\alpha p-1 >0$. This estimate is sufficient to apply Kolmogorov's continuity theorem.

\begin{proposition}\label{Pro:Properties}
The processes $\lb\gBm\rb_{t\geq 0}$, $0<\alpha<2$, $0<\beta<1$, have the following properties:
\begin{enumerate}[label=\textit{(\roman*)}]
	\item $(B_t^{\alpha,\beta})_{t\geq 0}$ has covariance
	\[
		\E\lb\gBm B_s^{\alpha,\beta}\rb = \frac{1}{2\Gamma(\beta+1)}\lb t^\alpha+s^\alpha -\abs{t-s}^\alpha\rb =: \frac{\gamma_{\alpha,\beta}(t,s)}{2},\quad t,s\ge 0,
	\]
	and $\E((\gBm)^2) = \frac{1}{\Gamma(\beta+1)}t^\alpha$ for $t\ge 0$.
	\item The density of the finite dimensional distributions of $(\gBm)_{t\geq 0}$ with respect to the Lebesgue measure are given by
	\begin{align*}
		f_{\alpha,\beta}(x) = \frac{(2\pi)^{-n/2}}{\sqrt{\Gamma(1+\beta)^n \det\gamma_{\alpha,\beta}}} \int_0^\infty \frac{M_\beta(\tau)}{\tau^{n/2}} \exp\lb-\halb\frac{x^T\gamma_{\alpha,\beta}^{-1}x}{\tau\Gamma(1+\beta)}\rb\,\rmd\tau,
	\end{align*}
	where $x\in\Rn$ and $\gamma_{\alpha,\beta}$ denotes the matrix $\lb\gamma_{\alpha,\beta}(t_i,t_j)\rb_{i,j=1,\dots,n}\in\R^{n\times n}$. For the definition of the $M$-function see Appendix \ref{Sec:H-function}, below.
	\item $(\gBm)_{t\geq 0}$ has stationary increments.
\end{enumerate}
\end{proposition}

\begin{proof}
The covariance can be calculated using Lemma \ref{Lemma:sps} and Lemma \ref{Lemma:isometry} and Example 3.1 in \cite{MM09}. Obviously we can deduce that $\E\lb (\gBm)^2\rb = \frac{1}{\Gamma(\beta+1)}t^\alpha$. The second assertion follows by the same arguments as in \cite{MP08}, Proposition 1. Proposition 3.2 in \cite{MM09} and Lemma \ref{Lemma:sps} show that
\begin{align*}
	&\E\lb\e^{i\theta(B^{\alpha,\beta}_{t+h}-\gBm)}\rb =\E\lb\e^{i\theta B^{\alpha,\beta}_h}\rb.
\end{align*}
This shows the third assertion.
\end{proof}

\begin{remark}
Consider the scaled ggBm $\hat{B}^{\alpha,\beta}_t := B^{\alpha,\beta}_{2^{1/\alpha}t}$. Then a comparison of Proposition \ref{Pro:Properties} with Proposition 2 in \cite{MP08} shows that the scaled version of Definition \ref{Def:gBm} coincides with the definitions of \cite{MP08,MM09}. Consequently $\hat{B}^{\beta,\beta}$ is a grey Brownian motion as in \cite{Sch92}.
\end{remark}

\begin{remark}
Generalized grey Brownian motion is a generalization of well known stochastic processes. Choosing $\beta=1$ then Proposition \ref{Pro:Properties} shows that $\E_{\mu_1}(B_t^{\alpha,1} B_s^{\alpha,1}) = 1/2((t^\alpha + s^\alpha - \abs{t-s}^\alpha)$ and the finite dimensional distributions of $B_t^{\alpha,1}$ are given by
\[
	f_{\alpha,1} (x) = \frac{1}{\sqrt{2\pi \det\gamma_{\alpha,1}}^n} \exp\lb -\halb x^T\gamma_{\alpha,1}^{-1}x\rb.
\]
Thus $B_t^{\alpha,1}$ is a fractional Brownian motion with Hurst parameter $H=\alpha/2$. Consequently the process $B_t^{1,1}$ is a Brownian motion. 
\end{remark}

\begin{remark}
A random vector $X=(X_1,\ldots,X_n)^T$, $n\in\N$, is said to have an elliptical distribution with parameter $m\in\Rn$, symmetric and positive definite $\Sigma\in\R^{n\times n}$ and some characteristic function $\phi\colon\R\to\R$ (for short $X\sim \mathrm{EC}_n(m,\Sigma,\phi)$)  if its characteristic function is of the form 
\[
	\E(\e^{\ii\theta^T X}) = \e^{\ii\theta^Tm} \phi(\theta^T\Sigma\theta),\quad \theta\in\Rn,
\]
see e.g.~equation (2.11) in \cite{FKN90}. The proof of Proposition \ref{Pro:Properties} shows that $X=(B_{t_1}^{\alpha,\beta},\ldots,B_{t_n}^{\alpha,\beta})$ has an elliptical distribution and $X\sim \mathrm{EC}_n(0,\Gamma(\beta+1)\gamma_{\alpha,\beta}, \rmE_\beta(-\cdot))$, see also \cite{DSE13}.
\end{remark}

\section{The time-fractional heat equation}\label{Sec:Heateq}
The well-known heat equation
\begin{align}\label{eq:heateq}
	\partt u(t,x) &= \halb \partxx u(t,x), \\
	u(0,x) &= u_0(x), \nonumber
\end{align}
for $t\ge 0$, $x\in\R$, and initial value $u_0\in\cS(\R)$ can be generalized to a fractional differential equation in two ways. Replacing the time derivative by the Caputo derivative of order $\alpha$ denoted by $^C\!D^{2\alpha}_{0+}$, $0<\alpha\leq 1$, we obtain the equation
\begin{align}\label{eq:Meq}
	\lb^C\!D^{\alpha}_{0+} u\rb(t,x) = \halb \partxx u(t,x), \quad x\in\R,\;t\geq 0.
\end{align}
Starting from the heat equation in integral from and replacing the integral by a Riemann-Liouville fractional integral gives the equation
\begin{equation}\label{eq:Seq}
	u(t,x) = u_0(x) + \frac{1}{\Gamma(\alpha)} \int_0^t (t-s)^{\alpha-1} \partxx u(s,x)\,\rmd s.
\end{equation}
\eqref{eq:Meq} has been solved for example in \cite{Mai95} and a solution to \eqref{eq:Seq} is given in \cite{SW89}. The Green's function $G$ has the series expansion
\begin{align}\label{eq:SGreen}
	&G(t,x,\alpha) =\frac{1}{\sqrt{4\pi t^\alpha}}\sum_{k=0}^\infty \frac{(-1)^k}{k!} \frac{\Gamma(1/2-k)}{\Gamma(1-\alpha k-\alpha/2)}\lb\frac{r^2}{4t^\alpha}\rb^k \nonumber\\ &\quad\quad + \frac{r}{4t^\alpha\sqrt{\pi}}\sum_{k=0}^\infty \frac{(-1)^k}{k!} \frac{\Gamma(-1/2-k)}{\Gamma(1-\alpha k-\alpha)}\lb\frac{r^2}{4t^\alpha}\rb^k.
\end{align}
The same solution has also been given in \cite{Koch90}. Note that in general fractional differential equations with Caputo derivative are equivalent to the fractional integral equations using the Riemann-Liouville fractional integral if the solution is absolutely continuous, see Theorem \ref{Theo:CDI}, below.

For the heat equation \eqref{eq:heateq} the solution is given by the Feynman-Kac formula, i.e.~$u(t,x) = \E(u_0(x+B_t))$, for twice differentiable initial value $u_0$ with compact support and a Brownian motion $B$. \cite{Sch92} showed that $u(t,x) = \E(u_0(x+B_t^\alpha))$ solves the fractional heat equation \eqref{eq:Seq} for $u_0\in\cS(\R)$ and grey Brownian motion $(B_t^\alpha)_{t\geq 0} := (B_t^{\alpha,\alpha})_{t\geq 0}$. We prove in the following that, replacing $u_0$ by $\delta_x$, a Green's function to \eqref{eq:Seq} is obtained. For the proof we need the following lemma:

\begin{lemma}\label{Lem:eqEbeta}
For all $\lambda\in\C$ and for all $\beta\in(0,1)$ the Mittag-Leffler function $\rmE_\beta$ satisfies the following equation:
\[
	\rmE_\beta\lb-\halb\lambda^2 t^\alpha\rb = 1- \halb\lambda^2 \frac{1}{\Gamma(\beta)} \int_0^t (t^{\alpha/\beta} - s^{\alpha/\beta})^{\beta-1} \frac{\alpha}{\beta} s^{\alpha/\beta-1} \rmE_\beta\lb - \halb\lambda^2 s^\alpha\rb\,\rmd s.
\]
\end{lemma}

\begin{proof}
First a simple calculation shows that
\begin{equation}\label{eq:MLequation}
	\rmE_\beta\lb-\halb\lambda^2 t^\beta\rb = 1- \halb\lambda^2 \frac{1}{\Gamma(\beta)} \int_0^t (t - s)^{\beta-1} \rmE_\beta\lb - \halb\lambda^2 s^\beta\rb\,\rmd s.
\end{equation}
Indeed, insert the series expansion of $\rmE_\beta$ in the right hand side. Due to the absolute convergence of the series we may interchange sum and integral. For the general case $\alpha\neq \beta$ note that
\[
	\rmE_\beta\lb-\halb\lambda^2 t^\alpha\rb = \rmE_\beta\lb -\halb\lambda^2 (t^{\alpha/\beta})^\beta\rb.
\]
Applying \eqref{eq:MLequation} yields
\begin{equation}\label{eq:fracintML}
	\rmE_\beta\lb-\halb\lambda^2 t^\alpha\rb = 1- \halb\lambda^2 \frac{1}{\Gamma(\beta)} \int_0^{t^{\alpha/\beta}} (t^{\alpha/\beta} - s)^{\beta-1} \rmE_\beta\lb - \halb\lambda^2 s^\beta\rb\,\rmd s.
\end{equation}
Now the coordinate transform $s=u^{\alpha/\beta}$ yields the desired result.
\end{proof}

\begin{remark}
Equation \eqref{eq:fracintML} shows that
\begin{equation}\label{eq:MLrelation}
	\rmE_\beta\lb-\halb\lambda^2 t^\alpha\rb = 1 - \halb\lambda^2 \lb I^{\beta}_{0+}\rmE_\beta\lb-\halb\lambda^2(\cdot)^\beta\rb\rb(t^{\alpha/\beta}).
\end{equation}
\end{remark}

\begin{theorem}\label{Theo:FFK}
Let $0<\beta<1$ and $0<\alpha<2$. For $x,y\in\R$ and $t> 0$ define $K(t,x,y) := \E_{\mu_\beta}\lb \delta_y(x+B_t^{\alpha,\beta})\rb$. Then $K$ is a Green's function to the equation
\[
	u(t,x) = u_0(x) + \halb \lb I^\beta_{0+} \partxx u((\cdot)^{\beta/\alpha},x)\rb(t^{\alpha/\beta}),\quad t>0,\;x\in\R,
\]
with initial value $u_0\in L^1(\R,\rmd x)\cap L^2(\R,\rmd x)$ satisfying
\begin{align*}
	&\int_\R \abs{(\cF^{-1}u_0)(\lambda)\lambda}^{1+\varepsilon} \,\rmd\lambda <\infty \quad\text{and}\quad\int_\R \abs{(\cF^{-1}u_0)(\lambda)\lambda^2}^{1+\varepsilon}\,\rmd\lambda <\infty,
\end{align*}
for some $\varepsilon>0$.
\end{theorem}

\begin{proof}
Define $u(t,x) = \int_\R u_0(y) K(t,x,y)\,\rmd y$ for $x\in\R$ and $t>0$. Note that 
\[
	\E_{\mu_\beta}\lb\delta_y(x+B_t^{\alpha,\beta})\rb = \frac{1}{2\pi} \int_\R \e^{\ii\lambda(x-y)}\rmE_\beta\lb-\halb\lambda^2t^\alpha\rb\,\rmd\lambda,
\]
see Proposition 5.6 in \cite{GJRS14}. Equation \eqref{eq:LaplaceMbeta} in the appendix below implies that
\[
	\int_\R \abs{u_0(y) \E_{\mu_\beta}\lb \delta_y(x+B_t^{\alpha,\beta})\rb}\,\rmd y \leq \frac{1}{\sqrt{2\pi t^\alpha}} \int_\R \abs{u_0(y)} \int_0^\infty M_\beta(r) r^{-1/2}\,\rmd r\,\rmd y.
\]
%
%
Since $\int_0^\infty M_\beta(r) r^{-1/2}\,\rmd r = \pi/\Gamma(1-\alpha/2)$, see e.g.~\cite{MMP10}, we obtain
\begin{align*}
	\int_\R \abs{u_0(y) \E_{\mu_\beta}\lb \delta_y(x+B_t^{\alpha,\beta})\rb}\,\rmd y \leq  \frac{1}{\sqrt{2t^\alpha}\Gamma(1-\alpha/2)} \int_\R \abs{u_0(y)}\,\rmd y <\infty.
\end{align*}
Hence $u$ is well defined and it holds
\begin{align}
	u(t,x) &= \frac{1}{\sqrt{2\pi}} \int_\R \e^{\ii\lambda x} \rmE_\beta\lb-\halb\lambda^2 t^\alpha\rb \lb\frac{1}{\sqrt{2\pi}} \int_\R u_0(r) \e^{-\ii\lambda y}\,\rmd y\rb\,\rmd\lambda \nonumber \\
	&= \frac{1}{\sqrt{2\pi}} \int_\R \e^{\ii\lambda x} \rmE_\beta\lb-\halb\lambda^2 t^\alpha\rb \lb\cF^{-1}u_0\rb(\lambda)\,\rmd\lambda.
\end{align}
Denote the integrand by $f(x,\lambda)$, $x,\lambda\in\R$. Then $f(x,\cdot)\in L^1(\R)$ for all $x\in\R$ and $f(\cdot,\lambda)$ is differentiable for all $\lambda\in\R$ with 
\[
	\abs{\partx f(x,\lambda)} = \abs{\lambda} \abs{\rmE_\beta\lb-\halb\lambda^2 t^\alpha\rb \lb\cF^{-1}u_0\rb(\lambda)}:=g(\lambda).
\]
We show that $g\in L^1(\R,\rmd x)$. With equation \eqref{eq:LaplaceMbeta} below we estimate $g$ as follows:
\begin{align*}
	\int_\R g(\lambda)\,\rmd\lambda \leq \int_0^\infty M_\beta(r) \int_\R \abs{(\cF^{-1}u_0)(\lambda) \lambda} \exp\lb-\halb\lambda^2 t^\alpha r\rb\,\rmd\lambda\,\rmd r. 
\end{align*} 
Next we use H\"older's inequality. Choose $p=1+1/\varepsilon$ and $q=p/(p-1)=1+\varepsilon$. Then $1/p+1/q=1$ and 
\begin{align*}
	&\int_\R \abs{(\cF^{-1}u_0)(\lambda) \lambda} \exp\lb-\halb\lambda^2 t^\alpha r\rb\,\rmd\lambda \\
	&\leq \lb\int_\R \abs{(\cF^{-1}u_0)(\lambda) \lambda}^{1+\varepsilon}\,\rmd\lambda\rb^{\tfrac{1}{1+\varepsilon}} \lb\int_\R \exp\lb-\halb p\lambda^2t^\alpha r\rb\,\rmd\lambda\rb^{1/p} \\
	&= \lb\int_\R \abs{(\cF^{-1}u_0)(\lambda) \lambda}^{1+\varepsilon}\,\rmd\lambda\rb^{\tfrac{1}{1+\varepsilon}} \sqrt{\frac{2\pi}{p t^\alpha}}^{1/p} r^{-1/(2p)}.
\end{align*}
Due to the condition on $u_0$ together with Lemma A.2 in \cite{GJRS14} it follows that $g\in L^1(\R,\rmd x)$ and we may interchange derivative and integral and obtain
\begin{align*}
	&\partx \int_\R \e^{\ii\lambda x} \rmE_\beta\lb-\halb\lambda^2 t^\alpha\rb \lb\cF^{-1}u_0\rb(\lambda)\,\rmd\lambda \\ 
	&\quad\quad = \int_\R \ii\lambda \e^{\ii\lambda x} \rmE_\beta\lb-\halb\lambda^2 t^\alpha\rb \lb\cF^{-1}u_0\rb(\lambda)\,\rmd\lambda.
\end{align*}
With similar arguments we can show that also the second derivative interchanges with the integral. Thus
\[
	\partxx u(t,x) = -\frac{1}{\sqrt{2\pi}} \int_\R \lambda^2 \e^{\ii\lambda x} \rmE_\beta\lb-\halb\lambda^2 t^\alpha\rb (\cF^{-1}u_0)(\lambda)\,\rmd\lambda.
\]
The assertion follows by using Lemma \ref{Lem:eqEbeta} that
\end{proof}

\begin{remark}
The series expansion of $u$ for the case $\alpha=\beta$ is calculated in Appendix \ref{Sec:H-function}, below, as:
\begin{align*}
	u(t,x)& = \frac{1}{\sqrt{2\pi t^\alpha}} \sum_{k=0}^\infty \frac{(-1)^k}{k!} \frac{\Gamma(1/2-k)}{\Gamma(1-1/2\alpha-\alpha k)} \lb\frac{x^2}{2t^\alpha}\rb^k  \\ &\quad\quad + \frac{x}{2t^\alpha\sqrt{\pi}} \sum_{k=0}^\infty \frac{(-1)^k}{k!} \frac{\Gamma(-1/2-k)}{\Gamma(1-\alpha-\alpha k)}\lb\frac{x^2}{2t^\alpha}\rb^k, \quad t>0,\;x\in\R.
\end{align*}
Comparing this to \eqref{eq:SGreen} we see that our approach is equivalent to \cite{Mai95}.
\end{remark}

\begin{remark}\label{Rem:FFK}
Define for $t>0$ and $x\in\R$ the function $u(t,x) = \E_{\mu_\beta}(\delta(x+B_t^{\alpha,\beta}))$. Then it can be shown that $u$ solves the equation
\[
	u(t,x) = \delta_0(x) + \halb \int_0^{t^{\alpha/\beta}} \frac{(t^{\alpha/\beta}-s)^{\beta-1}}{\Gamma(\beta)} \partxx u(s^{\beta/\alpha},x)\,\rmd s
\]
in the distribution space $\cS'(\R)$. Note that 
\begin{align*}
	u(t,x) &= \frac{1}{2\pi} \int_\R \e^{\ii\lambda x}\rmE_\beta\lb-\halb\lambda^2t^\alpha\rb\,\rmd\lambda.
\end{align*}
Inserting the result from Lemma \ref{Lem:eqEbeta} we achieve:
\begin{align*}
	u(t,x) &= \frac{1}{2\pi} \int_\R \e^{\ii\lambda x}\lb 1- \frac{\lambda^2}{2}\int_0^{t^{\alpha/\beta}} \frac{(t^{\alpha/\beta}-s)^{\beta-1}}{\Gamma(\beta)} \rmE_\beta\lb-\halb\lambda^2s^\beta\rb\,\rmd s \rb \,\rmd\lambda \\
	&= \frac{1}{2\pi} \int_\R \e^{\ii\lambda x}\,\rmd\lambda \\ &\quad\quad\quad + \halb\int_0^{t^{\alpha/\beta}} \frac{(t^{\alpha/\beta}-s)^{\beta-1}}{\Gamma(\beta)} \partxx\lb\frac{1}{2\pi}\int_\R \e^{\ii\lambda x}\rmE_\beta\lb -\halb\lambda^2s^\beta\rb\,\rmd\lambda\rb\,\rmd s \\
	&= \delta_0(x) + \halb \int_0^{t^{\alpha/\beta}} \frac{(t^{\alpha/\beta}-s)^{\beta-1}}{\Gamma(\beta)} \partxx u(s^{\beta/\alpha},x)\,\rmd s.
\end{align*}
The calculation and the involving divergent integrals can be treated rigorously by applying them to a test function $\xi\in\cS(\R)$.
\end{remark}

\begin{remark}
For the special case of fractional Brownian motion we set $\beta=1$. Then we obtain that 
\[
	K(t,x,y) = \E_{\mu_1} \lb \delta_y(x+B_t^{\alpha,1})\rb,\quad t>0,\;x,y\in\R,
\]
is a solution to
\begin{align*}
	u(t,x) &= u_0(x) + \halb \int_0^{t^{\alpha}} \partxx u(s^{1/\alpha},x)\,\rmd s  \\
	&= u_0(x) + \halb \int_0^t \alpha s^{\alpha-1} \partxx u(s,x) \,\rmd s,
\end{align*}
or equivalently
\begin{equation}\label{eq:eqfbm}
	\partt u(t,x) = \halb \alpha t^{\alpha-1} \partxx u(t,x).
\end{equation}
In the case of the Gaussian measure $\mu_1$ the $T_{\mu_1}$-transform of Donsker's delta can be calculated explicitly: 
\begin{align*}
	K(t,x,y) &= \lb T_{\mu_1} \delta_y(x+B_t^{\alpha,1}) \rb (0) 
	= \frac{1}{\sqrt{2\pi t^\alpha}} \exp\lb-\frac{(x-y)^2}{2t^\alpha}\rb.
\end{align*}
It can be easily verified that $K$ is indeed a solution to \eqref{eq:eqfbm}.
\end{remark}

\section{Grey noise process}\label{Sec:gnp}
We have seen in the previous section that generalized grey Brownian motion $B_t^{\alpha,\beta}$ plays a central role in the context of the time-fractional heat equation. This motivates the further analysis in the following section. We first calculate the $S_{\mu_\beta}$-transform of $B_t^{\alpha,\beta}$. Then we prove that the local times of generalized grey Brownian motion exist.

In the following we calculate the $S_{\mu_\beta}$-transform of a generalized grey Brownian motion $B_t^{\alpha,\beta} = \lab\cdot,M^{\alpha/2}_-\ind\rab\in L^2(\mu_\beta)$ for $\varphi = \varphi_1+i\varphi_2\in\cU_\beta \subset \cS_\C(\R)$. Here, $\cU_\beta = \lcb \varphi\in\cN_\C \mid \abs{\lab\varphi,\varphi\rab} < \varepsilon_\beta \rcb$, where $\varepsilon_\beta >0$ is chosen such that $\rmE_\beta(z) > 0$ for all $\abs{z}<\varepsilon_\beta$, $z\in\C$. Then the $S_{\mu_\beta}$-transform of $B_t^{\alpha,\beta}$ is well-defined and it holds that
\begin{align*}
	\lb S_{\mu_\beta} \gBm \rb (\varphi) &= \frac{1}{\rmE_\beta\lb\halb\lab\varphi,\varphi\rab\rb} \int_{\cS'(\R)} \lab\omega,M^{\alpha/2}_-\ind\rab\e^{\lab\omega,\varphi\rab}\,\rmd\mu_\beta(\omega) \\
	&= \frac{1}{\rmE_\beta\lb\halb\lab\varphi,\varphi\rab\rb} \int_{\cS'(\R)} \frac{\rmd}{\rmd s} \e^{\lab\omega,\varphi\rab+s\lab\omega,M^{\alpha/2}_-\ind\rab} \Big|_{s=0}\,\rmd\mu_\beta(\omega).
\end{align*}
For $\omega\in\cS'(\R)$ and $s\in [-1,1]$ denote $f(\omega,s) := \exp(\lab\omega,\varphi+sM^{\alpha/2}_-\ind\rab)$. Then $f$ is integrable with respect to $\omega$ and differentiable with respect to $s$ and 
\[
	\frac{\rmd}{\rmd s} f(\omega,s) = \lab\omega,M^{\alpha/2}_-\ind\rab \e^{\lab\omega,\varphi+sM^{\alpha/2}_-\ind\rab}.
\]
Moreover we have for all $s\in [-1,1]$ and for all $\omega\in\cS'(\R)$ the estimate
\begin{align*}
	\abs{\frac{\rmd}{\rmd s}f(\omega,s)} &\leq \abs{\lab\omega,M^{\alpha/2}_-\ind\rab} \e^{\lab\omega,\varphi_1\rab} \e^{\abs{\lab\omega,M^{\alpha/2}_-\ind\rab}} \\ 
	&\leq \e^{\lab\omega,\varphi_1\rab} \exp\lb 2\abs{\lab\omega,M^{\alpha/2}_-\ind\rab}\rb =:g(\omega).
\end{align*}
It is shown in Lemma 4.1 in \cite{GJRS14} that both mappings $\exp(\lab\cdot,\varphi_1\rab)$ and $\exp\lb 2\abs{\lab\cdot,M^{\alpha/2}_-\ind\rab}\rb$ are in $\in L^2(\mu_\beta)$. Hence we have that $g\in L^1(\mu_\beta)$ by the H\"older inequality. Thus interchanging derivative and integral is possible. Since $\mu_\beta$ satisfies the assumption (A1) on page \pageref{Assumptions} we get
\begin{align*}
	&\lb S_{\mu_\beta} B_t^{\alpha,\beta}\rb (\varphi) \\ &\quad\quad = \frac{1}{\rmE_\beta\lb\halb\lab\varphi,\varphi\rab\rb}\frac{\rmd}{\rmd s} \rmE_\beta\lb \halb \lab\varphi+sM^{\alpha/2}_-\ind,\varphi+sM^{\alpha/2}_-\ind\rab\rb\Big|_{s=0},
\end{align*}
see Corollary 4.3 in \cite{GJRS14}. Using Lemma \ref{derEbeta} and evaluating at $s=0$ we obtain
\[
	\lb S_{\mu_\beta} B_t^{\alpha,\beta}\rb (\varphi) = \lab\varphi,M^{\alpha/2}_-\ind\rab \frac{\rmE_{\beta,\beta}\lb\halb\lab\varphi,\varphi\rab\rb}{\beta \rmE_\beta(\halb\lab\varphi,\varphi\rab)}.
\]
We may use the integration by parts formula \eqref{eq:Intbyparts1} in the appendix below in the case $1<\alpha<2$ since $\varphi$ and $\ind$ are in every $L^p(\R,\rmd x)$, $p\geq 1$. In the case $0<1<\alpha$ we are allowed to use \eqref{eq:Intbyparts2} in the appendix below, since $D^{(1-\alpha)/2}_\pm\varphi\in L^2(\R,\rmd x)$ and $D^{(1-\alpha)/2}_\pm\ind\in L^p(\R,\rmd x)$ for $p=2/(1+2\alpha)$ by Corollary \ref{Cor:integrable}, below. Thus we get
\begin{align}\label{eq:gBmStrafo}
	\lb S_{\mu_\beta} B_t^{\alpha,\beta}\rb (\varphi) &= \lab M^{\alpha/2}_+\varphi,\ind\rab \frac{\rmE_{\beta,\beta}\lb\halb\lab\varphi,\varphi\rab\rb}{\beta \rmE_\beta(\halb\lab\varphi,\varphi\rab)} \nonumber\\
	&= \frac{\rmE_{\beta,\beta}\lb\halb\lab\varphi,\varphi\rab\rb}{\beta \rmE_\beta(\halb\lab\varphi,\varphi\rab)} \int_0^t \lb M^{\alpha/2}_+\varphi\rb(x)\,\rmd x.
\end{align}

\begin{remark}
In the white noise case $\alpha=\beta=1$ we obtain
\[
	\lb S_{\mu_1} B_t^{1,1}\rb(\varphi) = \lab\varphi,\ind\rab = \int_0^t \varphi(x) \,\rmd x
\]
which is the $S$-transform of Brownian motion.
\end{remark}

\begin{theorem}\label{Theo:Diff}
Each generalized grey Brownian motion $(0,\infty)\ni t\mapsto \gBm\in (\cS)^{-1}_{\mu_\beta}$ for $0<\alpha<2$, $0<\beta<1$, is differentiable, i.e.
	\[
		\lim_{h\to 0} \frac{B^{\alpha,\beta}_{t+h}-\gBm}{h}
	\]
	exists and converges to some element in $(\cS)^{-1}_{\mu_\beta}$ denoted by $N_t^{\alpha,\beta}$ fulfilling
	\[
		\lb S_{\mu_\beta} N_t^{\alpha,\beta}\rb(\varphi) = \lb M^{\alpha/2}_+\varphi\rb(t) \frac{\rmE_{\beta,\beta}\lb\halb\lab\varphi,\varphi\rab\rb}{\beta \rmE_\beta(\halb\lab\varphi,\varphi\rab)},
	\]
for every $\varphi$ from a suitable neighborhood $\cU\subset\cN_\C$ of zero.
\end{theorem}

\begin{proof}
Consider for $n\in\N$ and $t>0$
\[
	\Phi_n := \frac{B^{\alpha,\beta}_{t+h_n}-\gBm}{h_n}\in (\cS)^{-1}_{\mu_\beta},
\]
for a sequence $(h_n)_{n\in\N}$ such that $h_n\to 0$. By \eqref{eq:gBmStrafo}
\[
	S_{\mu_\beta}\Phi_n(\varphi) =  \frac{1}{h_n} \int_t^{t+h_n} \lb M^{\alpha/2}_+\varphi\rb(x)\,\rmd x \frac{\rmE_{\beta,\beta}\lb\halb\lab\varphi,\varphi\rab\rb}{\beta \rmE_\beta(\halb\lab\varphi,\varphi\rab)},
\]
for every $\varphi\in\cU_\beta$. Since $M^{\alpha/2}_+\varphi$ is continuous, see \cite{Bend1}, we have 
\[
	\lim_{n\to \infty} \frac{1}{h_n} \int_t^{t+h_n} \lb M^{\alpha/2}_+\varphi\rb(x)\,\rmd x \frac{\rmE_{\beta,\beta}\lb\halb\lab\varphi,\varphi\rab\rb}{\beta \rmE_\beta(\halb\lab\varphi,\varphi\rab)} = \lb M^{\alpha/2}_+\varphi\rb(t) \frac{\rmE_{\beta,\beta}\lb\halb\lab\varphi,\varphi\rab\rb}{\beta \rmE_\beta(\halb\lab\varphi,\varphi\rab)}.
\]
Thus $(S_{\mu_\beta}\Phi_n(\varphi))_{n\in\N}$ is a Cauchy sequence. Since the Mittag-Leffler functions are holomorphic there are $p,q\in\N$ such that 
\[
	\abs{\frac{\rmE_{\beta,\beta}\lb\halb\lab\varphi,\varphi\rab\rb}{\beta \rmE_\beta(\halb\lab\varphi,\varphi\rab)}} \leq K,\quad \text{for all }\varphi\in U_{p,q}.
\]
Furthermore it follows from Theorem 2.3 in \cite{Bend1} that there is $p'\in\N$ such that for all $x\in\R$
\[
	\abs{\lb M^{\alpha/2}_+\varphi\rb(x)} \leq C\abs{\varphi}_{p'}.
\]
Here, $(\abs{\cdot}_p)_{p\in\N}$ denotes the family of norms which defines the topology of $\cS(\R)$, see Example \ref{ex:wnsetting}. Choose $p^*>\max(p,p')$. By the mean value theorem find $s\in[t,t+h]$ such that for all $\varphi\in U_{p^*,q}$ and for all $n\in\N$
\begin{align*}
	&\abs{\frac{S_{\mu_\beta}\lb B_{t+h_n}^{\alpha,\beta}-B_t^{\alpha,\beta}\rb(\varphi)}{h_n}} \leq \abs{\frac{\rmE_{\beta,\beta}\lb\halb\lab\varphi,\varphi\rab\rb}{\beta \rmE_\beta(\halb\lab\varphi,\varphi\rab)}}\, \abs{\lb M^{\alpha/2}_+\varphi\rb(s)} \leq K\,C\abs{\varphi}_{p'}.
\end{align*}
Finally, applying Theorem \ref{charconv}, we see that $\lb\Phi_n\rb_{n\in\N}$ converges to the distribution $N^{\alpha,\beta}_t$ in $(\cS)^{-1}_{\mu_\beta}$ fulfilling
\[
	\lb S_{\mu_\beta} N^{\alpha,\beta}_t\rb(\varphi) = \lim_{n\to\infty} S_{\mu_\beta}\Phi_n(\varphi),\quad \varphi\in U_{p^*,q}.
\]
\end{proof}

\subsection{Local times of generalized grey Brownian motion}
The ggBm local time, denoted by $L_{\alpha,\beta}(a,T)$, measures the time for $B_t^{\alpha,\beta}$ spending up to time $T>0$ in $a\in\R$. It is already shown in \cite{DSE13} that generalized grey Brownian motion admits a square integrable local time by checking Berman's criteria. In the paper \cite{DSE13} the local time is introduced as the Radon-Nikodym derivative of the occupation measure. In detail: Let $I$ be a measurable set in the interval $[0,T]$, $T>0$, and $f\colon I\to\R$ measurable. The occupation measure $\mu_f$ is defined as
\[
	\mu_f(B) = \int_I \1_B(f(s))\,\rmd s, \quad B\in\cB(\R).
\]
If $\mu_f$ is absolutely continuous with respect to the Lebesgue measure $\rmd x$ we denote the corresponding density by $L^f(\cdot,I)$, i.e.
\[
	\int_I \1_B(f(s))\,\rmd s = \mu_f(B) = \int_B L^f(x,I)\,\rmd x, \quad B\in\cB(\R).
\]
$L^f(\cdot,I)$ is called \emph{local time} of $f$. Informally the local time can be expressed via the Dirac delta distribution. In fact it holds that
\begin{equation}\label{eq:loctime}
	L^f(x,I) = \int_I \delta_x(f(s))\,\rmd s,
\end{equation}
see e.g.~\cite{Kuo}, page 273. 
Motivated by \eqref{eq:loctime} we give a second approach for constructing the ggBm local time $L_{\alpha,\beta}(a,T)$. For $T>0$ and $a\in\R$ we define
\[
	L_{\alpha,\beta}(a,T) := \int_0^T \delta_a(B_t^{\alpha,\beta})\,\rmd t \in (\cS)^{-1}_{\mu_\beta}
\]
 	
\begin{theorem}\label{Theo:LocalTimes}
The local time $L_{\alpha,\beta}(a,T)$ of generalized grey Brownian motion with $0<\alpha<2$, $0<\beta<1$, $0<T<\infty$ and $a\in\R$ exists as a weak integral in $(\cS)^{-1}_{\mu_\beta}$ with $T_{\mu_\beta}$-transform
\[
	\lb T_{\mu_\beta} L_{\alpha,\beta}(a,T)\rb(\xi) = \int_0^T \lb T_{\mu_\beta}\delta_a(\gBm)\rb(\xi)\,\rmd t
\]
for all $\xi$ from a suitable neighborhood $\cU\subset \cS_\C(\R)$ of zero.
\end{theorem}

\begin{proof}
Let $\xi\in\cS_\C(\R)$ with $\abs{\xi}_{L^2(\R,\rmd x)}<M$ for some $M<\infty$. According to Proposition 5.6 in \cite{GJRS14} with $\eta=M^{\alpha/2}_-\ind\in L^2(\R,\rmd x)$ we have
\begin{align*}
	&\lb T_{\mu_\beta}\delta_a(B_t^{\alpha,\beta})\rb(\xi) \\ &= \frac{1}{2\pi} \int_\R \exp(-\ii xa) \rmE_\beta\lb-\halb x^2t^\alpha - \halb\lab\xi,\xi\rab - x\lab\xi,M^{\alpha/2}_-\ind\rab\rb\,\rmd x.
\end{align*}
The same calculation as in the proof of Proposition 5.2 in \cite{GJRS14} shows that
\begin{align*}
	&\int_0^T \abs{T_{\mu_\beta}\delta_a(B_t^{\alpha,\beta})(\xi)}\,\rmd t \\
	&\leq \frac{1}{\sqrt{2\pi}} \int_0^T \int_0^\infty M_\beta(r) r^{-1/2} t^{-\alpha/2} \exp\lb-\halb r z(t,\xi) \rb \,\rmd r\,\rmd t,
\end{align*}
with $z(t,\xi) = \lab\xi_1,\xi_1\rab-\lab\xi_2,\xi_2\rab - \frac{\lab\xi_1,M^{\alpha/2}_-\ind\rab^2}{t^\alpha}>-M^2$ since $\abs{\xi}_{L^2(\R,\rmd x)}<M$ and thus
\begin{align*}
	&\int_0^T \abs{T_{\mu_\beta}\delta_a(B_t^{\alpha,\beta})(\xi)}\,\rmd t \\
	&\leq \frac{1}{\sqrt{2\pi}} \frac{2}{2-\alpha}T^{1-\alpha/2} \int_0^\infty M_\beta(r) r^{-1/2}  \exp\lb \frac{1}{2}M^2r \rb \,\rmd r.
\end{align*}
By Lemma A.2 in \cite{GJRS14} we get
\[
	\int_0^T \abs{T_{\mu_\beta}\delta_a(B_t^{\alpha,\beta})(\xi)}\,\rmd t \leq \frac{K}{\sqrt{2\pi}} \frac{2}{2-\alpha}T^{1-\alpha/2}
\]
for some constant $K<\infty$. The assertion follows by Theorem \ref{charint}.
\end{proof}

\begin{remark}
The expectation of $L_{\alpha,\beta}(0,T)$ can be calculated explicitly:
\begin{align*}
	&\E(L_{\alpha,\beta}(T,0)) = \int_0^T \E(\delta_0(B_t^{\alpha,\beta}))\,\rmd t \\ &= \int_0^T \frac{1}{\sqrt{2\pi t^\alpha}}\frac{1}{\Gamma(1-\beta/2)}\,\rmd t = \frac{1}{\Gamma(1-\beta/2)\sqrt{2\pi}} \frac{2}{2-\alpha} T^{1-\alpha/2}.
\end{align*}
This result extends the existence result for ggBm local time in \cite{DSE13}.
\end{remark}

\section*{Acknowledgments}
We would like to thank Maria Jo\~ao Oliveira and Rui Vilela Mendes for motivating and inspiring discussions. We also want to thank Jos\'e Luis da Silva and Felix Riemann for fruitful discussions. Florian Jahnert gratefully acknowledges financial support in the form of fellowships from the Department of Mathematics of the University of Kaiserslautern and from the ``Stipendienstiftung Rheinland-Pfalz''.

\newpage

\appendix

\section{Basics of fractional calculus}\label{Sec:Frac}
In this chapter we introduce the basic notations of fractional calculus, for a detailed overview and proofs we refer to \cite{SKM}. First of all we give the definition of the Riemann-Liouville integral and derivative. Besides that there are a couple of other definitions. Mostly common in physics is the fractional derivative in the sense of Caputo. We also consider the Marchaud fractional derivative. This derivative appears in the context of fractional Brownian motion in the white noise setting. In general, these three fractional derivatives are different. But we prove that they coincide on $\cS(\R)$.

\subsection{Riemann-Liouville fractional integral and derivative}
We first give the definitions of the fractional integral and derivative in the sense of Riemann-Liouville for functions $f$ which are defined on an interval $[a,b]$ for $-\infty < a<b<\infty$.

\begin{theorem}[\cite{SKM}, Theorem 3.5]\label{Theo:RLint}
The Riemann-Liouville fractional integrals of order $0<\alpha<1$ of $f\in L^p([a,b],\rmd x)$, defined by
\begin{align}
	\lb I^\alpha_{a+}f\rb(x) &:= \frac{1}{\Gamma(\alpha)} \int_a^x f(t) (x-t)^{\alpha-1}\,\rmd t,\quad x\in[a,b], \label{eq:lsInt}\\
	\lb I^\alpha_{b-}f\rb(x) &:= \frac{1}{\Gamma(\alpha)} \int_x^b f(t) (t-x)^{\alpha-1}\,\rmd t,\quad x\in[a,b], \tag{\ref{eq:lsInt}$'$} \label{eq:rsInt}
\end{align}
exist if $1\leq p<1/\alpha$. Moreover $I^\alpha_{a+}f \in L^q([a,b],\rmd x)$, where $1\leq q < \frac{p}{1-\alpha p}$. If $1<p<1/\alpha$ then $I^\alpha_{0+}f \in L^q([a,b],\rmd x)$ for $q=\frac{p}{1-\alpha p}$. \eqref{eq:lsInt} and \eqref{eq:rsInt} are called left-sided and right-sided Riemann-Liouville fractional integral, respectively.
\end{theorem}

\begin{theorem}[\cite{SKM}, Lemma 2.2]
Let $f\colon [a,b]\to\R$ be absolutely continuous. Then the left-sided (right-sided) Riemann-Liouville fractional derivative of order $0<\alpha<1$ of $f$ given by 
\begin{align}
	\lb D^\alpha_{a+}f\rb(x) := \frac{1}{\Gamma(1-\alpha)} \frac{\rmd}{\rmd x} \int_a^x f(t) (x-t)^{-\alpha}\,\rmd t,\quad x\in[a,b], \label{eq:lsRLDInt} \\
	\lb D^\alpha_{b-}f\rb(x) := \frac{-1}{\Gamma(1-\alpha)} \frac{\rmd}{\rmd x} \int_x^b f(t) (t-x)^{-\alpha}\,\rmd t,\quad x\in[a,b], \tag{\ref{eq:lsRLDInt}$'$} \label{eq:rsRLDInt}
\end{align}
exists. Moreover $D_{a+}^\alpha f, D^\alpha_{b-}f\in L^p([a,b]\,\rmd x)$ for $1\leq p<1/\alpha$ and
\begin{align*}
	\lb D^\alpha_{a+}f\rb(x) = \frac{1}{\Gamma(1-\alpha)}\lb\frac{f(a)}{(x-a)^\alpha} + \int_a^x f'(t) (x-t)^{-\alpha}\,\rmd t\rb, \\
	\lb D^\alpha_{b-}f\rb(x) = \frac{1}{\Gamma(1-\alpha)}\lb\frac{f(b)}{(b-x)^\alpha} - \int_x^b f'(t) (t-x)^{-\alpha}\,\rmd t\rb.
\end{align*}
\end{theorem}

\begin{remark}
For functions $f$ defined on the positive real line, the definitions \eqref{eq:lsInt} and \eqref{eq:lsRLDInt} are valid for all $x\in [0,\infty)$.
\end{remark}

Next we consider functions $f$ which are defined on the real line.

\begin{theorem}[\cite{SKM}, Theorem 5.3]\label{Def:fracint}
For every $f\in L^p(\R,\rmd x)$ with $1\leq p<1/\alpha$ the left-sided (right-sided) Riemann-Liouville fractional integral of order $0<\alpha<1$ is defined as follows: 
\begin{align}
	\lb I^\alpha_{+}f\rb(x) &:= \frac{1}{\Gamma(\alpha)} \int_{-\infty}^x f(t) (x-t)^{\alpha-1}\,\rmd t,\quad x\in\R, \label{eq:lsRL}\\
	\lb I^\alpha_{-}f\rb(x) &:= \frac{1}{\Gamma(\alpha)} \int_x^\infty f(t) (t-x)^{\alpha-1}\,\rmd t,\quad x\in\R. \tag{\ref{eq:lsRL}$'$} \label{eq:rsRL}
\end{align}
Moreover $I^\alpha_{\pm}\colon L^p(\R,\rmd x)\to L^q (\R,\rmd x)$ is continuous for $1<p<1/\alpha$ and $q=\tfrac{p}{1-\alpha p}$.
\end{theorem}

In order to define the fractional derivatives we need the following:
\begin{lemma}[\cite{Mish}, Lemma 1.1.2]
Let $1\leq p<1/\alpha$ and assume that $f\in I^\alpha_\pm \bigl(L^p(\R,\rmd x)\bigr)$. Then there is a unique $\varphi\in L^p(\R,\rmd x)$ such that $f(x)=\lb I^\alpha_{\pm}\varphi\rb(x)$. This $\varphi$ is called Riemann-Liouville fractional derivative of $f$, denoted by $D^\alpha_\pm f$ and given by:
\begin{align}
	\lb D^\alpha_{+}f\rb(x) = \frac{1}{\Gamma(1-\alpha)} \frac{\rmd}{\rmd x} \int_{-\infty}^x f(t) (x-t)^{-\alpha}\,\rmd t,\quad x\in\R, \label{eq:lsRLD} \\
	\lb D^\alpha_{-}f\rb(x) = \frac{-1}{\Gamma(1-\alpha)} \frac{\rmd}{\rmd x} \int_x^\infty f(t) (t-x)^{-\alpha}\,\rmd t,\quad x\in\R. \tag{\ref{eq:lsRLD}$'$} \label{eq:rsRLD}
\end{align}
\end{lemma}

\begin{remark}\label{Rem:RLprop}
\begin{enumerate}
	\item By the definition of the Riemann-Liouville fractional derivative it holds automatically 	that
	\[
		D^\alpha_\pm I^\alpha_\pm f = f, \quad\text{for all } f\in L^p(\R,\rmd x), 1\leq p<1/\alpha.
	\]
	Moreover, if $f\in I^\alpha_\pm \bigl(L^p(\R,\rmd x)\bigr)$, $1\leq p<1/\alpha$, then we have also 
	\[
		I^\alpha_\pm D^\alpha_\pm f = f.
	\]	
	\item Equation (5.16) in \cite{SKM} proves the following integration-by-parts formula:
	Let $f\in L^p(\R,\rmd x)$ and $g\in L^q(\R,\rmd x)$ with $p,q\geq 1$ and $\frac{1}{p}+\frac{1}{q} = 1+\alpha$. Then it holds for all $0<\alpha<1$: 
	\begin{equation}\label{eq:Intbyparts1}
		\int_\R \lb I^\alpha_{+}f\rb (x) g(x) \,\rmd x = \int_\R f(x) \lb I^\alpha_{-}g\rb (x)\, \rmd x.
	\end{equation}
	\item The next formula gives the integration-by-parts formula for the fractional derivative, see (5.17) in \cite{SKM}: Let $0<\alpha<1$, $D^\alpha_-f\in L^p(\R,\rmd x)$, $D^\alpha_+g\in L^q(\R,\rmd x)$ with $1/p+1/q=1+\alpha$ and $f\in L^r(\R,\rmd x)$,  $g\in L^s(\R,\rmd x)$ with $1/r=1/p+\alpha$, $1/s = 1/q -\alpha$. Then, 
	\begin{equation}\label{eq:Intbyparts2}
		\int_\R \lb D^\alpha_- f \rb (x) g(x)\, \rmd x = \int_\R f(x) \lb D^\alpha_+ g \rb (x)\, \rmd x.
	\end{equation}
\end{enumerate}
\end{remark}

As an example we state the fractional integral and derivative of the indicator function $\1_{[a,b)}$, see Lemma 1.1.3 in \cite{Mish}:

\begin{example}\label{Ex:indicator}
The indicator function $\1_{(a,b)}$ for $a,b\in\R$ lies in the domain of $I^\alpha_\pm$. Moreover
\begin{align*}
	\lb I^\alpha_\pm\1_{(a,b)}\rb (t) = \frac{1}{\Gamma(\alpha+1)}\lb \mp(b-t)_\mp^\alpha  \pm(a-t)^\alpha_\mp \rb.
\end{align*}
The Riemann-Liouville fractional derivative of the indicator function $\1_{(a,b)}$ is given by 
\[
	\lb D^\alpha_\pm\1_{(a,b)}\rb (t) = \frac{1}{\Gamma(1-\alpha)}\lb \mp(b-t)_\mp^{-\alpha} \pm (a-t)^{-\alpha}_\mp \rb.
\]
\end{example}

\begin{lemma}\label{Lemma:integrable}
For $0<\alpha<1$ it holds that $D^\alpha_\pm\1_{[a,b)}\in L^1(\R,\rmd x)$.
\end{lemma}

\begin{proof}
See also \cite{ST94}. 
For $x<a$ and for some $\xi\in[a,b]$ we have by the mean value theorem
\[
	\abs{(b-x)^{-\alpha}_+ - (a-x)^{-\alpha}_+} = (b-a) \alpha (\xi-x)^{-\alpha-1} \leq (b-a) \alpha (a-x)^{-\alpha-1}.
\]
Thus
\[
	\int_{-\infty}^{a-1} \abs{(b-x)^{-\alpha}_+ - (a-x)^{-\alpha}_+}\,\rmd x \leq (b-a)\alpha \int_{-\infty}^{a-1} (a-x)^{-\alpha-1}\,\rmd x <\infty.
\]
For $x\in[a-1,a+\delta]$, $\delta>0$ small enough such that $a+\delta<b$, there is a constant $C<\infty$ such that $(b-x)^{-\alpha} \leq C$ and we get
\begin{align*}
	&\int_{a-1}^{a+\delta} \abs{(b-x)^{-\alpha}_+ - (a-x)^{-\alpha}_+}\,\rmd x \leq C(1+\delta) + \int_{a-1}^a (a-x)^{-\alpha}\,\rmd x < \infty.
\end{align*}
For $x\in[a+\delta,b]$, we only have to consider the first term and find that
\[
	\int_{a+\delta}^b (b-x)^{-\alpha}\,\rmd x < \infty.
\]
For $x>b$ the function is zero. This shows that $D^\alpha_-\1_{[a,b)}\in L^1(\R,\rmd x)$. Use similar arguments to get $D^\alpha_+\1_{[a,b)}\in L^1(\R,\rmd x)$. 
\end{proof}

\begin{corollary}\label{Cor:integrable}
A slight modification of the previous proof shows that $D^\alpha_\pm\1_{[a,b)}\in L^p(\R,\rmd x)$ whenever $\alpha p<1$.
\end{corollary}

\begin{lemma}\label{Lem:fracDerInd}
Let $0<\alpha<1/2$. Then for every $a,b\in\R$ it holds
\[
	\widetilde{D^\alpha_\pm\1_{[a,b]}}(x) = (\mp ix)^\alpha \widetilde{\1_{[a,b]}}(x).
\]
\end{lemma}

\begin{proof}
First let $0<\alpha<1$. From \cite{GS64} it is known that a Fourier transform pair is given in the distributional sense by
\[
	t_\pm^{\alpha-1} \quad \overset{\cF}{\longleftrightarrow} \quad \frac{\Gamma(\alpha)}{\sqrt{2\pi}}(\mp ix)^{-\alpha}.
\]
This means that for each $\varphi\in\cS(\R)$ it holds that
\[
	\lab \cF(\cdot)^{-\alpha}_\pm,\varphi\rab = \lab (\cdot)^{-\alpha}_\pm,\cF\varphi\rab = \frac{\Gamma(1-\alpha)}{\sqrt{2\pi}}\lab (\mp i\cdot)^{\alpha-1},\varphi\rab,
\]
i.e.
\begin{align*}
	&\int_\R t^{-\alpha}_\pm (\cF\varphi)(t)\,\rmd t = \frac{\Gamma(1-\alpha)}{\sqrt{2\pi}} \int_\R (\mp ix)^{\alpha-1}\varphi(x)\,\rmd x.
\end{align*}
Set $f(x):= -(b-x)^{-\alpha}_- + (a-x)^{-\alpha}_-$, $x\in\R$. Then Corollary \ref{Cor:integrable} shows that $f\in L^1(\R,\rmd x)\cap L^2(\R,\rmd x)$ for $0<\alpha<1/2$. Furthermore it holds for each $\varphi\in\cS(\R)$:
\begin{align*}
	&\lab \cF f,\varphi\rab = \lab f,\cF\varphi\rab = \int_\R f(x) (\cF\varphi)(x)\,\rmd x \\
	&= -\int_\R (b-x)^{-\alpha}_- (\cF\varphi)(x)\,\rmd x + \int_\R (a-x)^{-\alpha}_- (\cF\varphi)(x)\,\rmd x \\
	&= -\int_b^\infty (x-b)^{-\alpha} (\cF\varphi)(x) \,\rmd x + \int_a^\infty (x-a)^{-\alpha} (\cF\varphi)(x)\,\rmd x \\
	&= - \int_0^\infty t^{-\alpha} (\cF\varphi)(b+t)\,\rmd t + \int_0^\infty t^{-\alpha} (\cF\varphi)(a+t)\,\rmd t \\
	&= \int_\R t^{-\alpha}_+ \lb (\cF\varphi)(a+t) - (\cF\varphi)(b+t)\rb \,\rmd t \\
	&= \frac{1}{\sqrt{2\pi}} \int_\R \int_\R t^{-\alpha}_+ \varphi(s) (\e^{ias}-\e^{ibs})\e^{its}\,\rmd s\,\rmd t = \int_\R t^{-\alpha}_+ \cF\lb (\e^{ia\cdot}-\e^{ib\cdot})\varphi\rb(t)\,\rmd t \\
	&= \frac{\Gamma(1-\alpha)}{\sqrt{2\pi}} \int_\R (-ix)^{\alpha-1} (\e^{iax}-\e^{ibx})\varphi(x) \,\rmd x \\
	&= \frac{\Gamma(1-\alpha)}{\sqrt{2\pi}} \lab (-i\cdot)^{\alpha-1} (\e^{ia\cdot}-\e^{ib\cdot}),\varphi\rab.	
\end{align*}
Since $D^\alpha_+ \1_{[a,b]}(x) = \frac{1}{\Gamma(1-\alpha)} f(x)$, see Example \ref{Ex:indicator}, we conclude that
\[
	\widetilde{D^\alpha_+ \1_{[a,b]}}(x) = -\frac{1}{\sqrt{2\pi}} (-ix)^{\alpha-1} (\e^{ibx}-\e^{iax}).
\]
Since it is well known that
\[
	\widetilde{\1_{[a,b]}}(x) = \frac{1}{\sqrt{2\pi}} \frac{\e^{ibx}-\e^{iax}}{ix}
\]
we obtain
\[
	\widetilde{D^\alpha_+ \1_{[a,b]}}(x) = \widetilde{\1_{[a,b]}}(x) (-ix) (-ix)^{\alpha-1} = \widetilde{\1_{[a,b)}}(x)(-ix)^\alpha.
\]
A similar calculation for $g(x) = (b-x)_+^{-\alpha} - (a-x)_+^{-\alpha}$ shows that
\[
	\widetilde{D^\alpha_-\1_{[a,b]}}(x) = (ix)^\alpha \widetilde{\1_{[a,b]}}(x).
\]
\end{proof}

\subsection{Fractional derivative of Caputo-type}
In contrast to the Riemann-Liouville fractional derivatives \eqref{eq:lsRLDInt}, \eqref{eq:rsRLDInt} and \eqref{eq:lsRLD}, \eqref{eq:rsRLD} we denote the fractional derivative of Caputo-type by $^C\!D^\alpha$. 

\begin{theorem}[\cite{KST06}, Theorem 2.1]
A sufficient condition for the left- and right-sided Caputo fractional derivatives of $f\colon [a,b]\to\R$ of order $0<\alpha<1$ is that $f$ is absolutely continuous. Then
\begin{align}
	\lb^C\!D^\alpha_{a+}f\rb(x) := \lb D^\alpha_{a+}f\rb(x) - \frac{f(a)}{\Gamma(1-\alpha)}(x-a)^{-\alpha},\quad x\in[a,b], \label{eq:lsCDInt} \\
	\lb^C\!D^\alpha_{b-}f\rb(x) := \lb D^\alpha_{b-}f\rb(x) - \frac{f(b)}{\Gamma(1-\alpha)}(b-x)^{-\alpha},\quad x\in[a,b]. \tag{\ref{eq:lsCDInt}$'$} \label{eq:rsCDInt}
\end{align}
In this case it holds that \eqref{eq:lsCDInt} and \eqref{eq:rsCDInt} coincide with the following expressions:
\begin{align}
	\lb^C\!D^\alpha_{a+}f\rb(x) = \frac{1}{\Gamma(1-\alpha)} \int_a^x f'(t) (x-t)^{-\alpha}\,\rmd t,\quad x\in[a,b], \label{eq:lsCDInt2} \\
	\lb^C\!D^\alpha_{b-}f\rb(x) = \frac{-1}{\Gamma(1-\alpha)} \int_x^b f'(t) (t-x)^{-\alpha}\,\rmd t,\quad x\in[a,b]. \tag{\ref{eq:lsCDInt2}$'$} \label{eq:rsCDInt2}
\end{align}
\end{theorem}


\begin{theorem}[\cite{Diet10}, Theorem 3.7 and 3.8]\label{Theo:CDI}
Let $f$ be continuous on $[a,b]$. Then
\[
	^C\!D^\alpha_{a+}I^\alpha_{a+}f = f.
\]
If $f$ is absolutely continuous, then
\[
	\lb I^\alpha_{a+}{} ^C\!D^\alpha_{a+}f\rb(x) = f(x) - f(a), \quad x\in [a,b].
\]
\end{theorem}

%

For $f\colon\R\to\R$ the Caputo fractional derivatives on the real line are defined as follows, provided they exist:
\begin{align}
	\lb^C\!D^\alpha_{+}f\rb(x) = \frac{1}{\Gamma(1-\alpha)} \int_{-\infty}^x f'(t) (x-t)^{-\alpha}\,\rmd t,\quad x\in\R, \label{eq:lsCD} \\
	\lb^C\!D^\alpha_{-}f\rb(x) = \frac{-1}{\Gamma(1-\alpha)} \int_x^\infty f'(t) (t-x)^{-\alpha}\,\rmd t,\quad x\in\R. \tag{\ref{eq:lsCD}$'$} \label{eq:rsCD}
\end{align}
A sufficient condition for the existence of $^C\!D^\alpha_\pm f$ is that $f\in\cS(\R)$, see Theorem \ref{Theo:EqualityonS} below.

\subsection{Marchaud fractional derivative}
Next we define the left- and right-sided Marchaud fractional derivatives on the real line, provided they exist, for $f\colon\R\to\R$ by:
\begin{align}
	\lb^M\!D^\alpha_+f\rb(x) = \frac{\alpha}{\Gamma(1-\alpha)} \int_0^\infty \frac{f(x)-f(x-\xi)}{\xi^{\alpha+1}}\,\rmd\xi,\quad x\in\R, \label{eq:lsMD} \\
	\lb^M\!D^\alpha_-f\rb(x) = \frac{\alpha}{\Gamma(1-\alpha)} \int_0^\infty \frac{f(x)-f(x+\xi)}{\xi^{\alpha+1}}\,\rmd\xi,\quad x\in\R. \tag{\ref{eq:lsMD}$'$} \label{eq:rsMD}
\end{align}
We also introduce the truncated Marchaud fractional derivatives on the real line by
\begin{align}
	\lb^M\!D^\alpha_{+,\varepsilon}f\rb(x) = \frac{\alpha}{\Gamma(1-\alpha)} \int_\varepsilon^\infty \frac{f(x)-f(x-\xi)}{\xi^{\alpha+1}}\,\rmd\xi,\quad x\in\R, \label{eq:lsTMD} \\
	\lb^M\!D^\alpha_{-,\varepsilon}f\rb(x) = \frac{\alpha}{\Gamma(1-\alpha)} \int_\varepsilon^\infty \frac{f(x)-f(x+\xi)}{\xi^{\alpha+1}}\,\rmd\xi,\quad x\in\R. \tag{\ref{eq:lsTMD}$'$} \label{eq:rsTMD}
\end{align}
Then 
\[
	\lb^M\!D^\alpha_\pm f\rb = \lim_{\varepsilon\to 0} \lb^M\!D^\alpha_{\pm,\varepsilon}f\rb,
\]
where the limit is defined by the problem under consideration. 
For $f\colon[a,b]\to\R$ the Marchaud fractional derivative on the interval $[a,b]$ is defined for $x\in[a,b]$ by
\begin{align}
	\lb^M\!D^\alpha_{a+}f\rb(x) = \frac{f(x)(x-a)^{-\alpha}}{\Gamma(1-\alpha)} + \frac{\alpha}{\Gamma(1-\alpha)} \int_a^x \frac{f(x)-f(x-\xi)}{\xi^{\alpha+1}}\,\rmd\xi, \label{eq:lsMDInt} \\
	\lb^M\!D^\alpha_{b-}f\rb(x) = \frac{f(x)(b-x)^{-\alpha}}{\Gamma(1-\alpha)} + \frac{\alpha}{\Gamma(1-\alpha)} \int_x^b \frac{f(x)-f(x-\xi)}{\xi^{\alpha+1}}\,\rmd\xi. \tag{\ref{eq:lsMDInt}$'$} \label{eq:rsMDInt}
\end{align}

\begin{theorem}[\cite{SKM}, Theorem 13.1.]
Let $f\in L^p([a,b],\rmd x)$. Then
\[
	^M\!D^\alpha_{a+}I^\alpha_{a+}f = f.
\]
In particular, the Marchaud fractional derivative and Riemann-Liouville fractional derivative on an interval coincide for $f\in I^\alpha_{a+}\lb L^p([a,b],\rmd x)\rb$. This is true, if e.g. $f$ is absolutely continuous.
\end{theorem}

\begin{theorem}[\cite{SKM}, Theorem 6.1]
If $f\in L^p(\R,\rmd x)$ for $1\leq p<1/\alpha$. Then
\[
	\lb ^M\!D^\alpha_{\pm}I^\alpha_\pm f\rb(x) = f(x).
\]
Here $^M\!D^\alpha_{\pm}f$ is the $L^p$-limit of the truncated Marchaud fractional derivative \eqref{eq:lsTMD} and \eqref{eq:rsTMD}.
\end{theorem}


\begin{example}
The Marchaud fractional derivative of the indicator function $\1_{[a,b)}$ is given by:
\[
	\lb ^M\!D^\alpha_\pm\1_{[a,b)}\rb (t) = \frac{1}{\Gamma(1-\alpha)}\lb \mp(b-t)_\mp^{-\alpha} \pm (a-t)^{-\alpha}_\mp \rb.
\]
\end{example}

\begin{proof}
If $x>b$, then clearly $\lb ^M\!D^\alpha_-\1_{[a,b)}\rb (x)= 0$. Let now $x<a$. Then
\begin{align*}
	&\lb ^M\!D^\alpha_-\1_{[a,b)}\rb (x) = -\frac{\alpha}{\Gamma(1-\alpha)} \int_0^\infty \frac{\1_{[a,b)}(x+\xi)}{\xi^{\alpha+1}}\,\rmd\xi \\
	&= -\frac{\alpha}{\Gamma(1-\alpha)} \int_{a-x}^{b-x} \xi^{-\alpha-1}\,\rmd\xi = \frac{1}{\Gamma(1-\alpha)}\lb (b-x)^{-\alpha} - (a-x)^{-\alpha}\rb.
\end{align*}
For $x\in[a,b]$ it holds
\begin{align*}
	&\lb ^M\!D^\alpha_-\1_{[a,b)}\rb (x) = \frac{\alpha}{\Gamma(1-\alpha)} \int_0^\infty \frac{1 - \1_{[a,b)}(x+\xi)}{\xi^{\alpha+1}}\,\rmd\xi \\
	&= \frac{\alpha}{\Gamma(1-\alpha)} \int_{b-x}^\infty \xi^{-\alpha-1}\,\rmd\xi = \frac{1}{\Gamma(1-\alpha)} (b-x)^{-\alpha}.
\end{align*}
Altogether we obtain
\[
	\lb ^M\!D^\alpha_-\1_{[a,b)}\rb (t) = \frac{1}{\Gamma(1-\alpha)}\lb (b-t)_+^{-\alpha} - (a-t)^{-\alpha}_+ \rb.
\]
A similar calculation shows
\[
	\lb ^M\!D^\alpha_+\1_{[a,b)}\rb (t) = \frac{1}{\Gamma(1-\alpha)}\lb -(b-t)_-^{-\alpha} + (a-t)^{-\alpha}_- \rb.
\]
This finishes the proof.
\end{proof}

\subsection{Fractional derivatives of Schwartz test functions}
\begin{theorem}\label{Theo:EqualityonS}
The Riemann-Liouville fractional derivative coincides with the Marchaud fractional derivative and with the Caputo derivative on $\cS(\R)$.
\end{theorem}

\begin{proof}
The integral $\int_0^\infty f(x\mp s) s^{-\alpha}\,\rmd s$ for $f\in\cS(\R)$ exists since
\begin{align*}
	&\int_0^\infty \abs{f(x\mp s)} s^{-\alpha}\,\rmd s = \int_0^1 \abs{f(x\mp s)} s^{-\alpha}\,\rmd s + \int_1^\infty \abs{f(x\mp s)} s^{-\alpha}\,\rmd s \\
	&\leq \norm{f}_\infty \int_0^1 s^{-\alpha}\,\rmd s + \sup_{s\in\R} \abs{s f(s)} \int_1^\infty s^{-\alpha-1}\,\rmd s <\infty.
\end{align*}
The coordinate transformation $s=x-t$ yields
\begin{align*}
	\lb D^\alpha_+f\rb(x) &= \frac{1}{\Gamma(1-\alpha)}\frac{\rmd}{\rmd x}\int_{-\infty}^x f(t) (x-t)^{-\alpha}\,\rmd t \\ &= \frac{1}{\Gamma(1-\alpha)}\frac{\rmd}{\rmd x} \int_0^\infty f(x-s) s^{-\alpha}\,\rmd s.
\end{align*}
Interchange of derivative and integral is justified since $\abs{f'(x\mp s)}s^{-\alpha}$ is dominated by the integrable function
\[
	h(s)= \norm{f'}_\infty \1_{[0,1)}(s)s^{-\alpha} + \sup_{s\in\R} \abs{sf'(s)} \1_{[1,\infty)}(s) s^{-\alpha-1}.
\]
This yields
\[
	\lb D^\alpha_+f\rb(x) = \frac{1}{\Gamma(1-\alpha)} \int_0^\infty f'(x-s)s^{-\alpha}\,\rmd s.
\]
By similar arguments we obtain
\[
	\lb D^\alpha_-f\rb(x) = \frac{-1}{\Gamma(1-\alpha)} \int_0^\infty f'(x+s)s^{-\alpha}\,\rmd s.
\]
Comparing with \eqref{eq:lsCD} and \eqref{eq:rsCD} it follows that for all $f\in\cS(\R)$ 
\[
	D^\alpha_\pm f = {}^C\!D^\alpha_\pm f.
\]
Furthermore since $s^{-\alpha} = \alpha \int_s^\infty \xi^{-\alpha-1}\,\rmd\xi$:
\begin{align*}
	\lb D^\alpha_+ f\rb(x) &= \frac{\alpha}{\Gamma(1-\alpha)} \int_0^\infty f'(x- s)  \int_s^\infty \xi^{-\alpha-1}\,\rmd\xi\,\rmd s \\
	&= \frac{\alpha}{\Gamma(1-\alpha)} \int_0^\infty \int_0^\xi f'(x- s) \xi^{-\alpha-1}\,\rmd s\,\rmd\xi \\ 
	&= \frac{\alpha}{\Gamma(1-\alpha)} \int_0^\infty \frac{f(x)-f(x-\xi)}{\xi^{\alpha+1}}\,\rmd\xi.
\end{align*}
And analogously
\[
	\lb D^\alpha_+ f\rb(x) = \frac{\alpha}{\Gamma(1-\alpha)} \int_0^\infty \frac{f(x)-f(x+\xi)}{\xi^{\alpha+1}}\,\rmd\xi.
\] 
\eqref{eq:lsMD} and \eqref{eq:rsMD} shows that
\[
	D^\alpha_\pm f = {}^M\!D^\alpha_\pm f,\quad f\in\cS(\R).
\]
\end{proof}

\section{The H-function}\label{Sec:H-function}
The H-function was discovered by Charles Fox in 1961 \cite{Fox61} and is a generalization of the G-function of Meijer. The definition is as follows: Let $m,n,p,q\in\N$, $0\leq n\leq p$ and $1\leq m\leq q$. Let $A_i,B_j\in\R$ be positive and $a_i,b_j\in\R$ or $\C$ arbitrary for $1\leq i\leq p,1\leq j\leq q$. Then
\[
	H^{m\,n}_{p\,q}\lb z \left| \begin{matrix} (a_p,A_p) \\ (b_q,B_q) \end{matrix} \right. \rb = H^{m\,n}_{p\,q}\lb z \left| \begin{matrix} (a_1,A_1),\dots,(a_p,A_p) \\ (b_1,B_1),\dots,(b_q,B_q) \end{matrix} \right. \rb = \frac{1}{2\pi i}\int_\cL \Theta(s)z^{-s}\,\rmd s,
\]
where
\[
	\Theta(s) = \frac{\lb \prod_{j=1}^m \Gamma(b_j+sB_j) \rb \lb \prod_{j=1}^n \Gamma(1-a_j-sA_j)\rb}{\lb\prod_{j=m+1}^q \Gamma(1-b_j-sB_j)\rb\lb\prod_{j=n+1}^p \Gamma(a_j+sA_j)\rb}.
\]
For further details concerning the contour $\cL$ and existence of $H$ we refer to \cite{MSH10}. The series expansion of $H$ for $\abs{z}>0$ is given in \cite{PBM90}
\begin{multline}\label{eq:Hseries}
	H^{m\,n}_{p\,q}\lb z \left| \begin{matrix} (a_p,A_p) \\ (b_q,B_q) \end{matrix} \right. \rb = \sum_{i=1}^m \sum_{k=0}^\infty \frac{\prod_{j=1,j\neq i}^m \Gamma(b_j-(b_i+k)\tfrac{B_j}{B_i})}{\prod_{j=m+1}^q \Gamma(1-b_j+(b_i+k)\tfrac{B_j}{B_i})} \\ \times\frac{\prod_{j=1}^n \Gamma(1-a_j+(b_i+k)\tfrac{A_j}{B_i})}{\prod_{j=n+1}^p \Gamma(a_j-(b_i+k)\tfrac{A_j}{B_i})} \frac{(-1)^k z^{(b_i+k)/B_i}}{k!B_i},
\end{multline}
under the condition that $\sum_{j=1}^q B_j - \sum_{j=1}^p A_j >0$ and $B_k(b_j+l) \neq B_j(b_k+s)$ for $1\leq j,k \leq m, j\neq k$ and $l,s\in\N_0$.

\begin{remark}\label{Rem:Hprop}
A change of variables in the definition of the $H$-function shows the following properties, see also e.g.~\cite{KST06}:
\begin{enumerate}
	\item For all $z\in\C\setminus\lcb 0\rcb$ such that the $H$-function is defined the following inversion formula holds:
	\[
		H^{m\,n}_{p\,q}\lb z \left| \begin{matrix} (a_p,A_p) \\ (b_q,B_q) \end{matrix} \right. \rb = H^{n\,m}_{q\,p}\lb 1/z \left| \begin{matrix} (1-b_q,B_q) \\ (1-a_p,A_p) \end{matrix} \right. \rb.
	\]
	\item For all $\sigma\in\C$ and for all $z\in\C$ such that the $H$-function is defined there holds the formula
	\[
		z^\sigma H^{m\,n}_{p\,q}\lb z \left| \begin{matrix} (a_p,A_p) \\ (b_q,B_q) \end{matrix} \right. \rb = H^{m\,n}_{p\,q}\lb z \left| \begin{matrix} (a_p+\sigma A_p,A_p) \\ (b_q+\sigma B_q,B_q) \end{matrix} \right. \rb.
	\]
\end{enumerate}
\end{remark}

The $M$-Wright function $M_\beta$ for $0<\beta<1$ was introduced by Mainardi as an auxiliary function when finding the Green's function to the time-fractional diffusion-wave equation, see e.g.~\cite{Mai96, Mai96b}. Its series expansion is given by
\[
	M_\beta(z) = \sum_{n=0}^\infty \frac{(-z)^n}{n!\Gamma(-\beta n + 1-\beta)},\quad z\in\C.
\]
For all $z\in\C$ the following holds \cite{GJRS14}:
\begin{equation}\label{eq:LaplaceMbeta}
	\int_0^\infty M_\beta(r) \exp(-rz)\,\rmd r = \rmE_\beta(-z).
\end{equation}
For more details we refer to \cite{MMP10} and the references therein. 

Let us now compute the expectation of grey Donsker's delta:
\begin{align*}
	\E_{\mu_\beta}\lb\delta(a+B_t^\beta)\rb &= \lb T_{\mu_\beta}\delta(a+B_t^\beta)\rb(0) \\
	&= \frac{1}{2\pi} \int_\R \lb T_{\mu_\beta} \exp(ix(B_t^\beta+a))\rb(0)\,\rmd x \\
	&=\frac{1}{2\pi} \int_\R \e^{ixa} \rmE_\beta\lb-\halb x^2t^\beta\rb \,\rmd x \\
	&=\frac{1}{2\pi} \int_\R \e^{ixa} \int_0^\infty M_\beta(s) \exp\lb-\halb x^2t^\beta s\rb\,\rmd s\,\rmd x \\
	&= \frac{1}{2\pi} \int_0^\infty M_\beta(s) \int_\R \exp\lb-\halb x^2t^\beta s + ixa\rb\,\rmd x\,\rmd s \\
	&= \frac{1}{\sqrt{2\pi t^\beta}} \int_0^\infty M_\beta(s) s^{-1/2} \exp\lb -\frac{a^2}{2t^\beta}s^{-1}\rb\,\rmd s  .
\end{align*}
We express the Mainardi function and the exponential by H-functions. First we have 
\[
	M_\beta(s) = H^{1\,0}_{1\,1}\lb s \left| \begin{matrix} (1-\beta,\beta) \\ (0,1) \end{matrix} \right. \rb \quad \text{and} \quad	\e^{-z} = H^{1\,0}_{0\,1}\lb z \left| \begin{matrix} -- \\ (0,1) \end{matrix} \right. \rb.
\]
Thus we get
\begin{align*}
	\E_{\mu_\beta}\lb\delta(a+B_t^\beta)\rb = \frac{1}{\sqrt{2\pi t^\beta}} \int_0^\infty s^{1/2-1} &H^{1\,0}_{1\,1}\lb s \left| \begin{matrix} (1-\beta,\beta) \\ (0,1) \end{matrix} \right. \rb \\ &\quad\quad H^{1\,0}_{0\,1}\lb \frac{a^2}{2t^\beta}s^{-1} \left| \begin{matrix} -- \\ (0,1) \end{matrix} \right. \rb \,\rmd s.
\end{align*}
Using the inversion formula for the H-function, see Remark \ref{Rem:Hprop} we end up with
\begin{align*}
	\E_{\mu_\beta}\lb\delta(a+B_t^\beta)\rb = \frac{1}{\sqrt{2\pi t^\beta}} \int_0^\infty s^{1/2-1} &H^{1\,0}_{1\,1}\lb s \left| \begin{matrix} (1-\beta,\beta) \\ (0,1) \end{matrix} \right. \rb \\ &\quad\quad H^{0\,1}_{1\,0}\lb \frac{2t^\beta}{a^2}s \left| \begin{matrix} (1,1) \\ -- \end{matrix} \right. \rb \,\rmd s.
\end{align*}
Lemma 2 in \cite{MS69} shows that 
\begin{align*}
	&\int_0^\infty x^{\sigma-1} H^{m\,n}_{p\,q}\lb \alpha x \left| \begin{matrix} (a_p,A_p) \\ (b_q,B_q) \end{matrix} \right. \rb  H^{k\,l}_{r\,s}\lb \beta x \left| \begin{matrix} (c_r,C_r) \\ (d_s,D_s) \end{matrix} \right. \rb \, \rmd x \\
	&\quad\quad = \alpha^{-\sigma} H^{k+n\,l+m}_{q+r\,p+s}\lb \beta/\alpha \left| \begin{matrix} (1-b_q-B_q\sigma,B_q),(c_r,C_r) \\ (1-a_p-A_p\sigma,A_p),(d_s,D_s) \end{matrix} \right. \rb,
\end{align*}
under the following conditions
\begin{align*}
	&\Re(\sigma + b_j/B_j + d_i/D_i) > 0,\quad j=1,\dots,m,\;i=1,\dots,k, \\
	&\Re(\sigma + (a_j-1)/A_j + (c_j-1)/C_j) < 0,\quad j=1,\dots,n,\;i=1,\dots,l, \\
	&\lambda_1 = \sum_{j=1}^m B_j - \sum_{j=m+1}^q B_j + \sum_{j=1}^n A_j - \sum_{j=n+1}^p A_j > 0, \\
	&\lambda_2 = \sum_{j=1}^k D_j - \sum_{j=k+1}^s D_j + \sum_{j=1}^l C_j - \sum_{j=l+1}^r C_j > 0, \\
	&\abs{\arg\alpha} < \lambda_1\pi/2,\quad \abs{\arg\beta} < \lambda_2\pi/2.
\end{align*}
In our case we have $\sigma=1/2$ and $m=p=q=1$, $n=0$ and $\alpha=1$ and $k=s=0$, $l=r=1$ and $\beta=2t^\beta/a^2$ and $\lambda_1 = 1-\beta >0$ and $\lambda_2 = 1 >0$. Thus all conditions are fulfilled and we obtain
\begin{align*}
	&\E_{\mu_\beta}\lb\delta(a+B_t^\beta)\rb = \frac{1}{\sqrt{2\pi t^\beta}} H^{0\,2}_{2\,1}\lb \frac{2t^\beta}{a^2} \left| \begin{matrix} (1/2,1),(1,1) \\ (\beta/2,\beta) \end{matrix} \right. \rb .
\end{align*} 
By the inversion formula from Remark \ref{Rem:Hprop} it holds
\[
	\E_{\mu_\beta}\lb\delta_a(\lab\cdot,\eta\rab)\rb = \frac{1}{\sqrt{2\pi \lab\eta,\eta\rab}} H^{2\,0}_{1\,2}\lb \frac{a^2}{2\lab\eta,\eta\rab} \left| \begin{matrix} (1-\beta/2,\beta) \\ (1/2,1),(0,1)\end{matrix} \right. \rb .
\] 
The definition of the $H$-function implies 
\[
	\E_{\mu_\beta}\lb\delta_a(\lab\cdot,\eta\rab)\rb = \frac{1}{\sqrt{2\pi \lab\eta,\eta\rab}} H^{2\,0}_{1\,2}\lb \frac{a^2}{2\lab\eta,\eta\rab} \left| \begin{matrix} (1-\beta/2,\beta) \\ (0,1),(1/2,1)\end{matrix} \right. \rb .
\]
Applying \eqref{eq:Hseries} we finally have
\begin{align*}
	\E_{\mu_\beta}\lb\delta_a(\lab\cdot,\eta\rab)\rb &= \frac{1}{\sqrt{2\pi \lab\eta,\eta\rab}} \sum_{k=0}^\infty \frac{(-1)^k}{k!} \frac{\Gamma(1/2-k)}{\Gamma(1-\beta/2-\beta k)} \lb\frac{a^2}{2\lab\eta,\eta\rab}\rb^k  \\ &\quad\quad + \frac{a}{2\lab\eta,\eta\rab\sqrt{\pi}} \sum_{k=0}^\infty \frac{(-1)^k}{k!} \frac{\Gamma(-1/2-k)}{\Gamma(1-\beta-\beta k)}\lb\frac{a^2}{2\lab\eta,\eta\rab}\rb^k.
\end{align*}

\bibliographystyle{alpha}
\bibliography{/home/florian/Latex/bibliography}
 
\end{document}